\documentclass[11pt]{article}

\usepackage[utf8]{inputenc}
\usepackage[T1]{fontenc}

\usepackage[a4paper]{geometry}

\usepackage[colorlinks=false]{hyperref}

\usepackage{amsmath, amsfonts,amsthm,amssymb}
\usepackage{bbm}

\usepackage{graphicx}
\usepackage{epstopdf}

\usepackage{comment}

\usepackage{txfonts}

\usepackage[font=footnotesize,labelfont=bf]{caption}

\usepackage{color}

\usepackage{customCommands}

\theoremstyle{theorem}

\newtheorem{theorem}{Theorem}
\newtheorem{lemma}[theorem]{Lemma}
\newtheorem{proposition}[theorem]{Proposition}
\newtheorem{definition}[theorem]{Definition}

\newtheorem{myAlgo}{Algorithm}
\newtheorem{projProc}{Procedure}

\usepackage[sort&compress, numbers]{natbib}
\usepackage{snapshot}

\usepackage{booktabs}
\usepackage{multicol}

\usepackage[boxed, ruled, lined]{algorithm2e}

\newcommand{\MSSIM}{\ensuremath{\operatorname{MSSIM}}}
\renewcommand{\argmin}{\operatorname*{argmin}}

\usepackage[usenames,dvipsnames,svgnames]{xcolor}

\usepackage{blindtext}

\usepackage{enumerate}
\usepackage{csquotes}

\usepackage[font=footnotesize]{subcaption}
\usepackage{float}

\title{Iterative Potts minimization for the recovery of signals with discontinuities from indirect measurements -- the multivariate case}

\author{
	Lukas Kiefer, Martin Storath, Andreas Weinmann
}

\begin{document}
\maketitle

\newlength\figureheight
\newlength\figurewidth
\setlength\figureheight{0.15\textwidth}

\begin{abstract}
	Signals and images with discontinuities appear in many problems in such diverse areas as
	biology, medicine, mechanics, and electrical engineering. 
	The concrete data are often discrete, indirect and noisy measurements
	of some quantities describing the signal under consideration. 
	A frequent task is to find the segments of the signal or image which corresponds to finding the discontinuities or jumps in the data.
	Methods based on minimizing the piecewise constant Mumford-Shah functional 
	--whose discretized version is known as Potts energy-- 
	are advantageous in this scenario, in particular, in connection with segmentation.	
	However, due to their non-convexity, minimization of such energies is challenging.
	In this paper we propose a new iterative minimization strategy for
	the multivariate Potts energy dealing with indirect, noisy measurements.
	We provide a convergence analysis and underpin our findings with numerical experiments.  
\end{abstract}
\maketitle

\vspace{2ex} {\it Keywords}: Piecewise-constant Mumford-Shah model, Potts model, majorization-minimi\-zation methods, image segmentation,  joint reconstruction and segmentation, ill-posed inverse problems, Radon transform,  deconvolution.

\vspace{0.5ex} {\it AMS subject classifications.}
94A08, 68U10, 65D18, 
65K10,
90C26, 90C39.

\section{Introduction}

Problems involving reconstruction tasks for functions with discontinuities appear in various biological and medical applications.
Examples are
the steps in the rotation of the bacterial flagella motor \cite{sowa2005direct, sowa2008bacterial,nord2017catch},
the cross-hybridization of DNA \cite{snijders2001assembly,drobyshev2003specificity,hupe2004analysis}, 
x-ray tomography \cite{ramlau2007mumford},
electron tomography \cite{klann2011mumfordElectron} and SPECT \cite{klann2011mumford,wolf2013few}. 
An engineering example is crack detection in brittle material in mechanics \cite{artina2013linearly}. 
Further examples may for instance be found in the papers 
\cite{little2011generalized,little2011generalized2,chartrand2009fast, chan1993image, frick2014multiscale} and the references therein.
In general, signals with discontinuities appear in many applied problems.
A central task is to restore the jumps, edges, change points or segments of the signals or images from the observed data.
These observed data are usually indirectly measured. Furthermore, they consist of measurements on a discretized grid and are typically corrupted by noise.

In many scenarios, nonconvex nonsmooth variational methods are a suitable choice 
for the partitioning task, i.e., the task of finding the jumps/edges/change points;
{see for example \cite{little2011generalized,nord2017catch,boykov2001fast}.
In particular, methods based on piecewise constant Mumford-Shah functionals \cite{mumford1985boundary,mumford1989optimal} have been used in various different applications.
The piecewise-constant Mumford-Shah model also appears in statistics and image processing 
where it is often called \emph{Potts model} \cite{geman1984stochastic,boykov2001fast, boysen2009jump, boysen2009consistencies,pock2009convex,winkler2003image};
this is a tribute to Renfrey B. Potts and his work in statistical mechanics \cite{potts1952some}.
The variational formulation of the piecewise-constant Mumford-Shah/Potts model (with an indirect measurement term) is given by 
\begin{equation}\label{eq:pottsGeneralA}
\textstyle\argmin_u \ \gamma \, \| \nabla u\|_{0} + \norm{A u - f}_{2}^2.
\end{equation}
Here, $A$ is a linear operator modeling the measurement process, e.g.,~the Radon transform in computed tomography (CT), or the point-spread function of the microscope in  microscopy. 
Further,  $f$ is an element of the data space, e.g.,~a sinogram or part of it in CT, or the blurred microscopy image in microscopy.
The mathematically precise definition of the jump term $\| \nabla u \|_{0}$ in the general situation is rather technical.
However, if $u$ is piecewise-constant and the discontinuity set of $u$ is sufficiently regular, say, a union of $C^1$ curves, then $\| \nabla u\|_{0}$ is just the total arc length of this union.
In general, the gradient $\nabla u$ is given in the distributional sense and the boundary length is expressed in terms of the \mbox{$(d-1)$-dimensional} Hausdorff measure.
When $u$ is not piecewise-constant, the jump penalty is infinite \cite{ramlau2010regularization}.
The second term measures the fidelity of a solution $u$ to the data $f.$ 
The parameter $\gamma > 0$ controls the balance between data fidelity and jump penalty.
(A wider class of Mumford-Shah models can be obtained by replacing the 
squared $L^2$ distance by more general data terms such as other norm-based expressions or divergences.)

The piecewise-constant Mumford-Shah/Potts model can be interpreted in two ways.
On the one hand, if the imaged object is (approximately) piecewise-constant, 
then the solution is an (approximate) reconstruction of the imaged object.
On the other hand, since  a piecewise-constant solution 
directly induces a partitioning of the image domain,
it can be seen as joint reconstruction and segmentation.
Executing reconstruction and segmentation jointly typically leads to better results
than performing the two steps successively \cite{klann2011mumford,ramlau2007mumford,ramlau2010regularization,storath2015joint}.
We note that in order to deal with the discrete data, the energy functional is typically discretized; see Section~\ref{sec:Discretization}. Some references concerning Mumford-Shah functionals are \cite{ambrosio1990approximation,chambolle1995image,ramlau2010regularization,jiang2014regularizing, blake1987visual,fornasier2010iterative,nikolova2000thresholding} and also the references therein;
see also the book \cite{ambrosio2000functions}.
The piecewise constant Mumford-Shah functionals are among the maybe most well-known representatives of the class of free-discontinuity problems introduced by De Giorgi \cite{de1991free}.

The analysis of the nonsmooth and nonconvex problem \eqref{eq:pottsGeneralA} is rather involved. 
We discuss some analytic aspects.
We first note that without additional assumptions the existence of minimizers of \eqref{eq:pottsGeneralA} is not guaranteed in a continuous domain setting
\cite{fornasier2010iterative,fornasier2013existence, ramlau2010regularization, storath2013jump}.
To ensure the existence of minimizers,
additional penalty terms such as an $L^p$ ($1<p<\infty$) term of the form 
$\|u\|_p^p$ \cite{ramlau2007mumford, ramlau2010regularization}
or pointwise boundedness constraints \cite{jiang2014regularizing} have been considered.
We note that the existence of minimizers is guaranteed in the discrete domain setup for typical discretizations \cite{fornasier2010iterative, storath2013jump}.
Another important topic is to verify that
the Potts model is a regularization method in the sense of inverse problems.
The first work dealing with this task is \cite{ramlau2010regularization}: 
The authors assume that the solution space consists of non-degenerate piecewise-constant functions with at most $k$ (arbitrary, but fixed) different values 
which are additionally bounded. Under relatively mild assumptions on the operator $A$, they show stability.  Further, by giving a suitable parameter choice rule, they show that the method is a regularizer in the sense of inverse problems.
Related references are \cite{klann2013regularization,jiang2014regularizing}
with the latter including (non-piecewise-constant) Mumford-Shah functionals.
We note that Mumford-Shah approaches (including the piecewise constant Mumford-Shah variant) 
also regularize the boundaries of the discontinuity set of the underlying signal \cite{jiang2014regularizing}.

Solving the Potts problem is algorithmically challenging. 
For $A = \mathrm{id},$ it is NP-hard for multivariate domains \cite{veksler1999efficient, boykov2001fast}, 	and, for general linear operators $A,$ it is even NP-hard for univariate signals \cite{storath2013jump}.
Thus,  finding a global minimizer within reasonable time seems to be unrealistic in general.
Nevertheless, due to its importance,
many approximative strategies for multivariate Potts problems with $A = \mathrm{id}$ have been proposed.
(We note that the case $A = \mathrm{id}$ is important as well since it captures the partitioning problem in image processing.)
For the Potts problem with general $A$ there are still some but not that many existing approaches, 
in particular in the multivariate situation.
For a more detailed discussion, we refer to the paragraph on algorithms for piecewise constant Mumford-Shah problems below.
A further discussion of methods for reconstructing piecewise constant signals may be found in  \cite{little2011generalized2}.
In \cite{weinmann2015iterative}, we have considered the univariate Potts problem for a general operator $A$ and have proposed a majorization-minimization strategy which we called iterative Potts minimization in analogy to iterative thresholding schemes. In this work, we will develop iterative Potts minimization schemes for the more demanding multivariate situation which is important for multivariate applications as appearing in imaging problems.

\paragraph{Existing algorithmic approaches to the piecewise constant Mumford-Shah problem and related problems.} 
We start to consider the Potts problem for general operator $A.$
In \cite{bar2004variational}, Bar et al.~consider an Ambrosio-Tortorelli-type approximation. 
Kim et al. use a level-set based active contour method 
for deconvolution in \cite{kim2002curve}.
Ramlau and Ring \cite{ramlau2007mumford} employ a related level-set approach 
for the joint reconstruction and segmentation of x-ray tomographic images;
further applications are electron tomography \cite{klann2011mumfordElectron} and SPECT \cite{klann2011mumford}. 
The authors of the present paper have proposed a strategy based on the alternating methods of multipliers in \cite{storath2013jump} for the univariate case and in \cite{storath2015joint} for the multivariate case. 

Fornasier and Ward \cite{fornasier2010iterative} rewrite Mumford-Shah problems 
as a pointwise penalized problem 
and derive generalized iterative thresholding algorithms for the rewritten problems in the univariate situation.
Further, they show that their method converges to a local minimizer in the univariate case.
Their approach principally carries over to the piecewise constant Mumford-Shah functional as explained 
in \cite{storath2013jump,weinmann2015iterative} and then results in a $\ell^0$ sparsity problem.
In the univariate situation, this NP-hard optimization problem is unconstrained and may 
be addressed by iterative hard thresholding algorithms for $\ell^0$ penalizations, analyzed by Blumensath and Davies in \cite{blumensath2008iterative,blumensath2009iterative}. 
(Note that related algorithms based on iterative soft thresholding for $\ell^1$ penalized problems have been considered by Daubechies, Defrise, and De Mol in \cite{daubechies2004iterative}.)
Artina et al. \cite{artina2013linearly} in particular consider the multivariate discrete Mumford-Shah model using the pointwise penalization approach of \cite{fornasier2010iterative}. In the multivariate setting, this results in a corresponding nonconvex and nonsmooth problem with linear constraints. 
The authors successively minimize local quadratic and strictly convex perturbations
(depending on the previous iterate)
of a (fixed) smoothed version of the objective by
augmented Lagrangian iterations which themselves can be accomplished by 
iterative thresholding via a Lipschitz continuous thresholding function.
They show that the accumulation points of the sequences produced by their algorithm  
are constraint critical points of the smoothed problem.
In the multivariate situation, a similar approach for rewriting the Potts problem results in an $\ell^0$ sparsity problem with additional equality constraints.
Algorithmic approaches for such $\ell^0$ sparsity problem with equality constraints 
are the penalty decomposition methods of \cite{lu2012iterative, lu2013sparse, zhang2013ell0}.
The connection with iterative hard thresholding is that the inner loop of the employed two stage process usually is of iterative hard thresholding type. 
The difference of the hard thresholding based methods to our approach in this paper is that we do not have to deal with constraints and the full matrix $A$ but with the nonseparable regularizing term $\|\nabla u \|_0$ instead of its separable counterpart $\|u\|_0.$ Hence we cannot use hard thresholding. 

Another frequently appearing method in the context of restoration of piecewise constant images is total variation minimization \cite{rudin1992nonlinear}. 
There the jump penalty $\|\nabla u \|_0$ is replaced by the total variation $\|\nabla u \|_1.$ 
The arising minimization problem is convex and therefore numerically tractable with 
convex optimization techniques \cite{chambolle2011first, condat2012direct}.
Candès, Wakin, and Boyd  \cite{candes2008enhancing} use iteratively reweighted total variation minimization for piecewise constant recovery problems.
Results of compressed sensing type related to the Potts problem have been derived by Needell and Ward \cite{needell2013near,needell2013stable}:
under certain conditions, minimizers of the Potts function agree with total variation minimizers.
However, in the presence of noise, total variation minimizers might significantly differ from minimizers of the Potts problem. But the minimizers of the Potts problem are the results frequently desired in practice.
Further, algorithms based on convex relaxations of the Potts problem \eqref{eq:pottsGeneralA} have gained a lot of interest in recent years; 
see, e.g., \cite{lellmann2011continuous, bae2011global, chambolle2012convex, brown2012completely, goldluecke2013tight, strekalovskiy2012convex}.

We next discuss approaches for the multivariate Potts problem for the situation $A = \mathrm{id}$
which is particularly interesting in image processing and for which there are some further approaches.  
The first class of approaches is the approach via graph cuts.
Here, the range space of $u$ is a priori restricted to a relatively small number of values.
The problem remains NP-hard, but it then allows for
an approach by sequentially solving binary partitioning problems via minimal graph cut algorithms
\cite{boykov2001fast,kolmogorov2004energy,boykov2004experimental}.
We here point out that this approach also can deal with 
(possibly non-convex) data fidelity terms more general 
than the squared $L^2$ data term 
employed in \eqref{eq:pottsGeneralA} (in the case $A = \mathrm{id}$).
Another approach is to limit the number $k$ of different values which $u$ may take without discretizing the range space a priori.  
For $k=2,$ active contours were used by Chan and Vese \cite{chan2001active} to minimize the corresponding binary Potts model. 
They use a level set function to represent the partitions which evolves according to the Euler-Lagrange equations of the Potts model. A globally convergent strategy for the binary segmentation problem is presented in  \cite{chan2006algorithms}.
The active contour method for $k=2$ was extended to 
larger $k$ in \cite{vese2002multiphase}. 
Note that, for $k >2$ the problem is NP hard.
We refer to \cite{cremers2007review} for an overview on level set segmentation.
In \cite{hirschmuller2005accurate,hirschmuller2006stereo, hirschmuller2008stereo}, Hirschm\"uller proposes 
a non-iterative strategy for the Potts problem which is based on cost aggregation.  
It has lower computational cost, but comes with lower quality reconstructions compared with graph cuts.
Due to the small number of potential values of $u,$ these methods mainly appear in connection with image segmentation. 
Methods for restoring piecewise constant images without restricting the range space are proposed in Nikolova et al.~\cite{nikolova2008efficient, nikolova2010fast}. 
They  use non-convex regularizers which are algorithmically approached using a graduated non-convexity approach. 
We note that the Potts problem \eqref{eq:pottsGeneralA} does not fall into the class of problems considered in \cite{nikolova2008efficient, nikolova2010fast}.
Last but not least, Xu et al.~\cite{xu2011image} proposed 
a piecewise constant model reminiscent of the Potts model
that is approached by a half-quadratic splitting 
using a pixelwise iterative thresholding type technique.
It was later extended to a method for blind image deconvolution \cite{xu2013unnatural}.

\paragraph{Contributions.}

The contributions of this paper are threefold: 
{\em (i)} We propose a new iterative minimization strategy for multivariate piecewise constant Mumford-Shah/Potts objective functions  as well as a (still NP-hard) quadratic penalty relaxation. 
{\em(ii)} We provide a convergence analysis of the proposed schemes. 
{\em(iii)} We show the applicability of our schemes in several experiments. 

Concerning {\em (i)}, we propose two schemes which are based on majorization-minimization 
or forward-backward splitting methods of Douglas-Rachford type \cite{lions1979splitting}.
The one scheme addresses the Potts problem directly, whereas the other scheme treats a quadratic penalty relaxation. The solutions of the relaxed problem themselves are not feasible for the Potts problem but near to a feasible solution of the Potts problem where nearness can be quantified. In particular, when a given tolerance in applications is acceptable the relaxed scheme is applicable.	
In contrast to the approaches in \cite{fornasier2010iterative,blumensath2008iterative} and \cite{lu2012iterative, lu2013sparse} for sparsity problems
which lead to thresholding algorithms, our approach leads to non-separable yet computationally tractable problems in the backward step.

Concerning {\em (ii)}, we first analyze the proposed quadratic penalty relaxation scheme. In particular, we show convergence towards a local minimizer. 
Due to the NP hardness of the quadratic penalty relaxation, the convergence result is in the range of what can be expected best. Concerning the scheme for the non-relaxed Potts problem we also perform a convergence analysis.
In particular, we obtain results on the convergence towards local minimizers on subsequences. 
The quality of the convergence results is comparable with the ones in \cite{lu2012iterative, lu2013sparse}. We note that compared with \cite{lu2012iterative, lu2013sparse} 
we face the additional challenge to deal with the non-separability of the backward step.
(We note that in practice we observe convergence of the whole sequence, not on a subsequence.)

Concerning {\em (iii)} 
	we consider problems with full and partial data. We begin to apply our algorithms to deconvolution problems. In particular, we consider deblurring and denoising 
	Gaussian blur images and
	motion blur images, respectively. We further consider noisy and undersampled Radon data, together with the task of joint reconstruction, denoising and segmentation.
	Finally, we use our method in the situation of pure image partitioning (without blur)
	which is a widely considered problem in computer vision. 
	
\paragraph{Organization of the paper.}

In Section~\ref{sec:DerivationAlgorithm} we derive the proposed algorithmic schemes. 
In Section~\ref{sec:Analysis} we provide a convergence analysis for the proposed schemes. 
In Section~\ref{sec:NumericalExperiments} we apply the algorithms derived in the present paper to concrete reconstruction problems.
In Section~\ref{sec:Conclusion} we draw conclusions.

\section{Majorization-minimization algorithms for multivariate Potts problems}
\label{sec:DerivationAlgorithm}

\subsection{Discretization}\label{sec:Discretization}

We use the following finite difference type discretiziation 
of the multivariate Potts problem \eqref{eq:pottsGeneralA} given by
\begin{equation}
\label{eq:PottsDiscrete}
P_\gamma(u) = \| Au - f\|_2^2 + \gamma \sum _{s=1}^S \omega_s  \left\| \nabla_{a_s} u \, \right\|_0,
\end{equation}
where the $a_s \in \mathbb Z^2$ come from a finite set of directions and the symbol $\nabla_{a_s} u \, (i,j)$ denotes the directional difference $u_{(i,j)+a_s} - u_{i,j}$ with respect to the direction $a_s$ at the pixel $(i,j)$. 
The symbol $\| \nabla_{a_s} u \|_0 $ denotes the number of nonzero entries of  $\nabla_{a_s} u.$
The simplest set of directions consists of the
unit vectors $a_1=(0,1),$ $a_2=(1,0)$ along with unit weights.
Unfortunately, when refining the grid, this discretization converges to a limit that measures the boundary in terms
of the $\ell^1$ analogue of the Hausdorff measure \cite{chambolle1995image}.
The practical consequences  are unwanted block artifacts in the reconstruction (geometric staircasing).
More isotropic results are obtained by adding  the diagonals
 $a_3=(1,1),a_4=(1,-1)$ to the directions $a_1$ and $a_2;$
a near isotropic discretization can be achieved by extending this system by the knight moves 
$a_5=(1,2),a_6=(2,1),a_7=(1,-2),a_8=(2,-1).$
(The name is inspired by the possible moves of a knight in chess.)
Weights $\omega_s$ for the system $\{a_1,a_2,a_3,a_4\}$ of coordinate directions and diagonal directions 
can be chosen as $\omega_{s} = \sqrt{2}-1  $ for the coordinate part $s=1,2$ 
and $\omega_{s} = 1-\tfrac{\sqrt{2}}{2}$} for diagonal part $s=3,4$. 
When additionally adding knight-move directions, weights $\omega_s$ for the system $\{a_1,\ldots,a_8\}$
can be chosen as $\omega_{s} = \sqrt{5} - 2$ for the coordinate part $s=1,2,$ 
$\omega_{s} = \sqrt{5} - \frac{3}{2}\sqrt{2}$ for diagonal part $s=3,4$, 
and $\omega_{s} = \frac{1}{2}(1 + \sqrt{2} - \sqrt{5})  $ for diagonal part $s=5,\ldots,8.$
There are several ways to derive weights $\omega_s$ for the neighborhood systems:  the method of \cite{chambolle1999finite} is based on an optimization approach, the method of \cite{boykov2003computing} is based on the Cauchy-Crofton formula, and the approach of \cite{storath2015joint} 
	is based on equating the euclidean lengths of straight lines  and the lengths of their digital counterparts.
We note that for the system $\{a_1,a_2,a_3,a_4\}$ of coordinate directions and diagonal directions the weights of \cite{chambolle1999finite} and in \cite{storath2015joint} coincide; the weights displayed for the knight-move case above are the ones derived by the scheme in \cite{storath2015joint}. 
For further details we refer to these references. 

We record that the considered problem \eqref{eq:PottsDiscrete} has a minimizer.
\begin{theorem}\label{thm:ExistenceMinimizer}
	The discrete multivariate Potts problem \eqref{eq:PottsDiscrete} has a minimizer.
\end{theorem}

The validity of Theorem~\ref{thm:ExistenceMinimizer} can be seen by following the lines of the proof of  \cite[Theorem 2.1]{hohm2015algorithmic} where an analogous statement is shown for the (non-piecewise constant) Mumford-Shah problem.	

\paragraph{Vector-valued images.}
We briefly discuss the extension of \eqref{eq:PottsDiscrete} to vector-valued 
images and multi-channel data, e.g.,
(blurred) RGB color images.
To this end, we assume multi-channel data $f = (f_1,\ldots,f_C)$ 
consisting of $C$ channels and 
images $u = (u_1,\ldots,u_C)$.
In this situation, the role of the first summand on the right-hand side of \eqref{eq:PottsDiscrete} is taken by the channel-wise sum 
$\sum_{c=1}^C \| Au_c - f_c\|_2^2$.
The symbol $\nabla_{a_s} u(i,j)$ now denotes the vector 
of directional differences with entries
$u_{(i,j)+a_s,c} - u_{i,j,c}$, $c=1,\ldots,C$ and the entirety of these
vectors form the rows of $\nabla u_{a_s}$.
Consequently, $\| \nabla_{a_s} u \|_0$ denotes
the number of nonzero rows of $\nabla_{a_s}u$.
As a result, introducing a jump between two pixels in all channels 
has the same costs as opening a jump in a single channel only.
This enforces the jumps to be aligned across the channels which
is in contrast to a channel-wise application of 
the single-channel Potts model \eqref{eq:PottsDiscrete}.

\subsection{Derivation of the proposed algorithmic schemes}\label{sec:ProposedSchemes}

We start out with the discretization \eqref{eq:PottsDiscrete} of the multivariate Potts problem.
We introduce $S$ versions $u_1,\ldots,u_S$ of the target $u$ and 
link them via equality constraints in the following consensus form to obtain the problem
\begin{equation}
\label{eq:MinimizeMultVar}
	P_\gamma(u_1,\ldots,u_S) \to \min,     \qquad \text{ s.t. } \quad u_1 = \ldots = u_S,
\end{equation}
 where the function $P_\gamma(u_1,\ldots,u_S)$ of the $S$ variables $u_1,\ldots,u_S$ is given by
\begin{equation}
\label{eq:PottsDiscreteSvariables}
P_\gamma(u_1,\ldots,u_S) = \sum _{s=1}^S \frac{1}{S}  \left\| Au_s - f \right\|_2^2 + \gamma \sum _{s=1}^S \omega_s  \left\| \nabla_{a_s} u_s \, \right\|_0.
\end{equation}
Note that solving \eqref{eq:MinimizeMultVar} is equivalent to solving the discrete Potts problem \eqref{eq:PottsDiscrete}. 
Further, note that we have overloaded the symbol $P_\gamma$ which, for one argument $u,$ denotes the Potts function of \eqref{eq:PottsDiscrete} and for $S$ arguments $u_1,\ldots,u_S$ denotes the 
energy function of \eqref{eq:PottsDiscreteSvariables}; we have the relation $P_\gamma(u,\ldots,u)= P_\gamma(u)$.

\paragraph{A majorization-minimization approach to the quadratic penalty relaxation of the Potts problem.} 
The quadratic penalty relaxation of \eqref{eq:PottsDiscreteSvariables} is given by 
\begin{equation} 
\label{eq:MinimizeMultVarRelax}
   P_{\gamma, \rho}(u_1,\ldots,u_S)  = \sum _{s=1}^S \frac{1}{S}  \left\| Au_s - f \right\|_2^2 + \gamma \sum _{s=1}^S \omega_s  \left\| \nabla_{a_s} u_s \, \right\|_0  +  
   \rho \sum_{1\leq s < s' \leq S}  c_{s,s'} \  \| u_s - u_{s'}\|_2^2.
\end{equation} 
Here, the soft constraints which replace the equalities  $u_1 = \ldots = u_S$ are realized via the squared Euclidean norms $\sum_{1\leq s < s' \leq S}  c_{s,s'} \  \| u_s - u_{s'}\|_2^2,$ 
where the nonnegative numbers $c_{s,s'}$ denote weights (which may be set to zero if no direct coupling between the particular $u_s, u_{s'}$ is desired.)
The symbol $\rho$ denotes a positive penalty parameter promoting the soft constraint, i.e.,    
increasing $\rho$ enforces the $u_i$ to be closer to each other w.r.t.\ the Euclidean distance. 
We note that we later analytically quantify the size of $\rho$ which is necessary to obtain an a priori prescribed tolerance in the $u_i;$ see \eqref{eq:maxNormTol} below.
Frequently, we use the short-hand  notation
\begin{equation}\label{eq:rhossprimeIsrhotimesc}
\rho_{s,s'}= \rho \ c_{s,s'}.
\end{equation} 
Typical choices of the $\rho_{s,s'}$ are 
\begin{equation}\label{eq:ChoiceRhos}
	\rho_{s,s'} = \rho   \quad \text{ for all } s,s', \qquad \qquad
	\text{ or} \quad
	\rho_{s,s'} = \rho  \  \delta_{((s+1){\rm\,mod \,} S ), s'}\,,    
\end{equation}
i.e., the constant choice ($c_{s,s'}=1$), as well as the coupling between consecutive variables with constant parameter ($\delta_{s,t} =1$ if and only if $s=t,$ and $\delta_{s,t} =0$ otherwise.) 
We note that in these situations only one additional positive parameter $\rho$ appears, and that this parameter is tied to the tolerance one is willing to accept as a distance of the $u_i;$
see Algorithm~\ref{alg:Algo1}.

For the majorization-minimization approach, we derive a surrogate functional 
\cite{daubechies2004iterative} of the function 
$P_{\gamma, \rho}(u_1,\ldots,u_S)$ of \eqref{eq:MinimizeMultVarRelax}.
For this purpose, we introduce the block matrix $B$ and the vector $g$ given by 
\begin{equation}\label{eq:DefOfBandg}
B \ = \
\begin{pmatrix}
S^{-1/2}A   &  0           & & \cdots & & 0   \\
0           &  S^{-1/2}A   & & \cdots & & 0   \\
\vdots           &     &  & \ddots & &  \vdots  \\
0   & 0  & & \cdots & S^{-1/2}A  & 0 \\
0   & 0  & & \cdots & 0 & S^{-1/2}A  \vspace{.8em} \\   
\rho_{1,2}^{1/2}I   & -\rho_{1,2}^{1/2}I & 0      & \ldots & 0 & 0\\
\rho_{1,3}^{1/2}I   & 0  & -\rho_{1,3}^{1/2}I     & \ldots & 0 & 0\\
& \vdots   &   &   & \vdots    &   \\
\rho_{1,S}^{1/2}I   & 0 & 0 & \ldots & 0 & -\rho_{1,S}^{1/2}I \\
0   & \rho_{2,3}^{1/2}I & -\rho_{2,3}^{1/2}I & \ldots & 0 & 0 \\
&    \vdots &  &   & \vdots    &   \\
0   & \rho_{2,S}^{1/2}I & 0 & \ldots & 0 & -\rho_{2,S}^{1/2}I \\ 
&  &  & \vdots & &  \\
&  &  & \vdots &  & \\        
0   & 0 & 0 & \ldots & \rho_{S-1,S}^{1/2}I & -\rho_{S-1,S}^{1/2}I \\   
\end{pmatrix},
\qquad
g \ = \
\begin{pmatrix}
S^{-1/2}f\\
S^{-1/2}f\\
\vdots  \\
S^{-1/2}f    \\
S^{-1/2}f   
\vspace{.8em} \\
0 \\
0\\
\vdots  \\
0  \\
0   \\
\vdots \\
0  \\ 
\vdots\\
\vdots\\        
0\\   
\end{pmatrix}. \ 
\end{equation}
Here $I$ denotes the identity matrix and $0$ the zero matrix; The matrix $B$ has $S$ block columns and $S+S(S-1)/2$ block rows. Further, we introduce the difference operator $D$ given by
\begin{equation}
\label{eq:DefDiffOp}
D(u_1,\ldots,u_S) = 
\begin{pmatrix}
\nabla_{a_1} u_1 \\
\vdots  \\
\nabla_{a_S} u_S \\
\end{pmatrix}
\end{equation}
which applies the difference w.r.t. the $i$th direction to the $i$th component of $u.$
We employ the weights $\omega_1,$ $\ldots, \omega_S$ to define the quantity
$\|D(u_1,\ldots,u_S)\|_{0,\omega}$ which counts the weighted number of jumps by
\begin{equation}
\label{eq:DefDiffOpZeroNorm}
\|D(u_1,\ldots,u_S)\|_{0,\omega} = \sum _{s=1}^S \omega_s  \left\| \nabla_{a_s} u_s \, \right\|_0.
\end{equation}
With all this comprehensive notation at hand, we may rewrite 
the function of \eqref{eq:MinimizeMultVarRelax} as
\begin{equation}
\label{eq:RelaxedProblemInShortNotation}
P_{\gamma,\rho}(u_1,\ldots,u_S) = \left\| B(u_1,\ldots,u_S)^T - g \right\|_2^2 + \gamma \ \Big\| \ D(u_1,\ldots,u_S) \ \Big\|_{0,\omega}.
\end{equation}
Using the representation \eqref{eq:RelaxedProblemInShortNotation}, the surrogate functional  in the sense of \cite{daubechies2004iterative} of $P_{\gamma,\rho}$
is given by 
\begin{align}
\label{eq:DerSurrogateInShortNotation}
P_{\gamma, \rho}^{\rm surr}&(u_1,\ldots,u_S,v_1,\ldots,v_S) = 
\frac{1}{L_{\rho}^2}\left\| B(u_1,\ldots,u_S)^T - g \right\|_2^2 + \frac{\gamma}{L_{\rho}^2} \ \Big\| \ D(u_1,\ldots,u_S) \ \Big\|_{0,\omega}\\
& \quad - \frac{1}{L_{\rho}^2}\left\| B(u_1,\ldots,u_S)^T -  B(v_1,\ldots,v_S)^T \right\|_2^2 +
\left\| (u_1,\ldots,u_S)^T -  (v_1,\ldots,v_S)^T \right\|_2^2.
\notag 
\end{align}
Here $L_{\rho} \geq 1$ denotes a constant which is chosen larger than the spectral norm $\| B \|$ of $B$ 
(i.e., the operator norm w.r.t.\ the $\ell^2$ norm.) This scaling is made to ensure that $B/L_{\rho}$ is contractive. 
In terms of $A$ and the penalties $\rho_{s,s'},$ we require that 
\begin{equation} \label{eq:EstBgeneral}
	L_{\rho}^2 > \|A\|_2^2/S + 2 \max_{s \in \{1,\ldots,S \} } \sum_{s': s'\neq s}^{S} \rho_{s,s'}.
\end{equation} 
	For the particular choice $\rho_{s,s'} = \rho$
	as on the left-hand side of \eqref{eq:ChoiceRhos} we can choose $L_{\rho}^2$ smaller, i.e.,   
	$
	L_{\rho}^2  > \|A\|_2^2/S + S \rho.
	$
	For only coupling neighboring $u_s$ with the same constant $\rho$, i.e.,  
	the right-hand coupling of \eqref{eq:ChoiceRhos}, we have
	$
	L_{\rho}^2  > \|A\|_2^2/S + \alpha \rho,
	$
	where 
	$\alpha = 4,$ if $S$ is even, 
	and $\alpha = 2 - 2 \cos \left(\frac{\pi(S-1)}{S}\right)$ if $S$ is odd.
These choices ensure that $B/L_{\rho}$ is contractive by Lemma~\ref{lem:BbyLcontractive}.
Basics on surrogate functionals as we need them for this paper are gathered in Section~\ref{sec:MajoMin}.  
Further details on surrogate functionals can be found in \cite{daubechies2004iterative,blumensath2008iterative,blumensath2009iterative}.

Using elementary properties of the inner product 
shows that 
\begin{align}\label{eq:SurrInExplicitForm}
P_{\gamma, \rho}^{\rm surr} (u_1,\ldots,u_S,v_1,\ldots,v_S) =  
&\bigg\|
(u_1,\ldots,u_S)^T-  \bigg((v_1,\ldots,v_S)^T-\frac{1}{L_{\rho}^2} B^T (B(v_1,\ldots,v_S)^T-g) \bigg) 
\bigg\|^2_2 \notag \\
 &+ \frac{\gamma}{L_\rho^2}\Big\| D(u_1,\ldots,u_S) \Big\|_{0,\omega}
+ R(v_1,\ldots,v_S) ,
\end{align}
where $ R(v_1,\ldots,v_S) $ is a rest term which is irrelevant when minimizing $P_{\rm surr}$ w.r.t.\
$u_1,\ldots,u_S$ for fixed $v_1,\ldots,v_S.$ Writing this down 
in terms of the original system matrix $A$ and the data $f$  yields
\begin{align} \label{eq:MinimizeMultVarSurr}
P_{\gamma, \rho}^{\rm surr}&\left( u_1,\ldots,u_S,v_1,\ldots,v_S \right) \\ = & \sum _{s=1}^S  
\left[
  \left\| u_s - \left(v_s + \tfrac{1}{SL_{\rho}^2} A^\ast  f - 
  \tfrac{1}{S L_{\rho}^2} A^\ast A v_s - \sum_{s \neq s'}\tfrac{\rho_{s,s'}}{L_{\rho}^2}
   (v_s-v_{s'})  
   \right)\right\|_2^2 + \tfrac{\gamma \omega_s}{L_{\rho}^2}  \left\| \nabla_{a_s} u_s \, \right\|_0
\right]  
+ R(v). \notag 
\end{align}
For the quadratic penalty relaxation of the Potts problem, i.e., for minimizing  the  problem  \eqref{eq:MinimizeMultVarRelax} we propose to use the surrogate iteration, i.e.
$u^{(n+1)}_1,\ldots,u^{(n+1)}_S$ 
$\in  \argmin_{u_1,\ldots,u_S} P_{\gamma, \rho}^{\rm surr} (u_1,\ldots,$ 
$u_S,u^{(n)}_1,\ldots,u^{(n)}_S).$
Applied to \eqref{eq:MinimizeMultVarSurr}, this surrogate iteration reads
\begin{equation}\label{eq:surIt4MultPottsRelaxed}
 \left(u^{(n+1)}_1,\ldots,u^{(n+1)}_S \right) 
 \in \argmin_{u_1,\ldots,u_S} 	
 \sum _{s=1}^S  
 \left[
 \left\| u_s - h^{(n)}_s \right\|_2^2 + \tfrac{\gamma \omega_s}{L_{\rho}^2}  \left\| \nabla_{a_s} u_s \, \right\|_0
 \right]  
\end{equation}
where $h^{(n)}_s$ is given by 
\begin{equation}\label{eq:dataInSurrIT}
	h^{(n)}_s = u^{(n)}_s + \tfrac{1}{SL_{\rho}^2} A^\ast  f - 
	\tfrac{1}{S L_{\rho}^2} A^\ast A u^{(n)}_s - \sum_{s':s' \neq s}\tfrac{\rho_{s,s'}}{L_{\rho}^2}
	(u^{(n)}_s-u^{(n)}_{s'}),  
	     \quad \text{ for all } s \in\{1,\ldots,S\}.
\end{equation}
Note that in Section~\ref{sec:ExactSubPotts} below, we derive an efficient algorithm which computes an exact minimizer of \eqref{eq:surIt4MultPottsRelaxed}.
Now assume that we are willing to accept a deviation between the $u_s$ which is small,
i.e.,  
\begin{equation}\label{eq:maxNormTol}
	 \|u_s - u_{s'}\|^2_2 = \sum_{i,j}|(u_s)_{ij} - (u_{s'})_{ij}|^2  < \tfrac{\varepsilon^2}{c_{s,s'}},  	 
\end{equation} 
for $\varepsilon>0$ and
for indices $s,s'$ with $c_{s,s'} \neq 0.$
The following algorithm computes a result fulfilling \eqref{eq:maxNormTol}.
\begin{myAlgo}\label{alg:Algo1}
  We consider the quadratic penalty relaxed Potts problem \eqref{eq:MinimizeMultVarRelax} and tolerance $\varepsilon$ for the targets $u_s$ we are willing to accept. We propose the following algorithm 
  for the  relaxed Potts problem \eqref{eq:MinimizeMultVarRelax} (which yields a result with targets $u_s$ deviating from each other by at most $\varepsilon/\sqrt{c_{s,s'}}$).
  \begin{itemize}
  	\item  Set $\rho$ according to \eqref{eq:chooseRho}, 
  	set $L_{\rho}$ according to \eqref{eq:EstBgeneral}
  	(or, in the special cases of \eqref{eq:ChoiceRhos}, as below \eqref{eq:chooseRho} and \eqref{eq:EstBgeneral}.)
  	\item [] Initialize $u^{(n)}_s$ as discussed in the corresponding paragraph below, (e.g., $u^{(n)}_s = 0$ for all $s.$) 
    \item Iterate until convergence:   
  \begin{align}\notag
  \text{ 1.}\quad & h^{(n)}_s = u^{(n)}_s + \tfrac{1}{SL_{\rho}^2} A^\ast  f - 
  \tfrac{1}{S L_{\rho}^2} A^\ast A u^{(n)}_s - \sum_{s':s' \neq s}\tfrac{\rho_{s,s'}}{L_{\rho}^2}
  (u^{(n)}_s-u^{(n)}_{s'}),  
  \quad s = 1,\ldots,S,  \\
  \label{eq:backwardStepAlg1}
  \text{ 2.}\quad & \left(u^{(n+1)}_1,\ldots,u^{(n+1)}_S \right) 
  \in \argmin_{u_1,\ldots,u_S} 	
  \sum _{s=1}^S  
  \left[
  \left\| u_s - h^{(n)}_s \right\|_2^2 + \tfrac{\gamma \omega_s}{L_{\rho}^2}  \left\| \nabla_{a_s} u_s \, \right\|_0
  \right].  
  \end{align}  
  \end{itemize}	
\end{myAlgo}

We will see in Theorem~\ref{thm:RelaxedConvergesAndTolerance} that this algorithm converges to a local minimizer of the quadratic penalty relaxation \eqref{eq:MinimizeMultVarRelax} and that the $u_s$ are $\varepsilon$-close, i.e., \eqref{eq:maxNormTol} is fulfilled.

\paragraph{The relation between the Potts problem and its quadratic penalty relaxation and 
obtaining a feasible solution for the Potts problem \eqref{eq:PottsDiscreteSvariables}
from the output of Algorithm~\ref{alg:Algo1}.
}

As pointed out above, we show in Theorem~\ref{thm:RelaxedConvergesAndTolerance} that Algorithm~\ref{alg:Algo1} produces a local minimizer of the quadratic penalty relaxation \eqref{eq:MinimizeMultVarRelax} of the Potts problem \eqref{eq:PottsDiscreteSvariables}
and that the corresponding variables of a resulting solution are close up to an a priori prescribed tolerance. This may in practice be already enough. 
However, strictly speaking  a local minimizer of the quadratic penalty relaxation \eqref{eq:MinimizeMultVarRelax} is not feasible for the Potts problem \eqref{eq:PottsDiscreteSvariables}.

We will now explain a projection procedure to derive a feasible solution for the Potts problem \eqref{eq:PottsDiscreteSvariables} from a local minimizer of \eqref{eq:MinimizeMultVarRelax} with nearby variables $u_s$ (as produced by Algorithm~\ref{alg:Algo1}.)
Related theoretical results are stated as Theorem~\ref{thm:EffectRelaxedAlgo4OriginalProblem}. 
In particular, we will see that in case the image operator $A$ is lower bounded, the projection procedure applied to the output of Algorithm~\ref{alg:Algo1} yields a feasible point which is close to a local minimizer of the original Potts problem \eqref{eq:PottsDiscreteSvariables}.

In order to explain the averaging procedure, we need some notions on partitionings.
Recall that a partitioning $\mathcal P$ consists of a (finite number of) segments $\mathcal P_i$  
which are pairwise disjoint sets of pixel coordinates whose union equals the image domain $\Omega,$ i.e.,
\begin{align}
	\cup_{i=1}^{N_\mathcal P} \mathcal P_i = \Omega, \qquad \mathcal P_i \cap \mathcal P_j = \emptyset \quad \text{for all $i,j=1,\ldots,N_\mathcal P.$}
\end{align} 
Here, we assume that each segment $\mathcal P_i$ is connected w.r.t.\ the neighborhood system $a_1,\ldots,a_S$
in the sense that there is a path connecting any two elements in $\mathcal P_i$ with steps in  
$a_1,\ldots,a_S.$

We will need the following proposed notion of a directional partitioning.
A directional partition w.r.t. a set of $S$ directions $a_1,\ldots,a_S$ consists of a set $\mathcal I$ of (discrete) intervals $I$,
where each interval $I$ is associated with exactly one of the directions $a_1,\ldots,a_S;$ 
here, an interval $I$ associated with the direction $a_s$ has to be of the form $I = \{(i,j)+ k a_s : k = 0,\ldots, K-1\},$
where $K \in \mathbb N$ and $I$ belongs to the discrete domain.
(For each direction $a_s$, the corresponding intervals form an ordinary partition.)
We note that Algorithm~\ref{alg:Algo1} which produces output $u=(u_1,\ldots,u_S):\Omega \to \mathbb R^s$ induces a directional partitioning as follows. We observe that each variable $u_s$ is associated with a direction $a_s.$ For any $s \in \{1,\ldots,S\},$ we let each (maximal) interval of constance of $u_s$ be an interval in $\mathcal I$ associated with $a_s.$

Each partitioning induces a directional partitioning $\mathcal I$ by letting the intervals $I$ of $\mathcal I$ be the stripes with direction $a_s$ obtained from segment $\mathcal P_i$
for each direction $s =1,\ldots, S$ 
and each segment $\mathcal P_i, i=1, \ldots, N_\mathcal P.$
Furthermore, each directional partitioning $\mathcal I$ induces a partitioning by the following merging process. 
\begin{definition}\label{def:EquivPath}	
	We say that pixels $x,y$ are related, in symbols, $x \sim y$, if there is a path $x_0=x,\ldots,x_N=y$ connecting $x,y$ in the sense that for any consecutive members  
	$x_i,x_{i+1},$ $i=1,\ldots,N-1,$ of the path there is an interval $I$ of the directional partitioning $\mathcal I$ containing both $x_i,x_{i+1}.$
\end{definition}
The relation $x\sim y$ obviously defines an equivalence relation and the corresponding equivalence classes $\mathcal P_i$
yield a partitioning on $\Omega.$ We use the symbols 
\begin{align}\label{eq:FromDirectionalPartitioning2Partitioning}
	\mathcal I(\mathcal P) = \mathcal I_{\mathcal P},\qquad  
	\mathcal P(\mathcal I) = \mathcal P_{\mathcal I},
\end{align}
to denote the mappings assigning a partitioning a directional partitioning and vice versa, respectively.

As a final preparation we consider a function  $u=(u_1,\ldots,u_S):\Omega \to \mathbb R^s$ as produced by Algorithm~\ref{alg:Algo1} and a partitioning $\mathcal P$ of $\Omega$ and define the following projection to a function $\pi_\mathcal P(u): \Omega \to \mathbb R$ by  
\begin{equation}\label{eq:DefProjectionPi}
	\pi_\mathcal P(u)|_{\mathcal P_i}
	 = \frac{\sum_{x \in \mathcal P_i}  \sum_{s = 1}^S  u_s(x) }{ S \ \#\mathcal P_i}, 	
\end{equation}
where the symbol $\#\mathcal P_i$ denotes the number of elements in the segment $\mathcal P_i.$
Hence the projection $\pi$ defined via \eqref{eq:DefProjectionPi} averages w.r.t.\ all components of $u$
and all members of the segment $\mathcal P_i$ and so produces a piecewise constant function w.r.t.\ the partitioning $\mathcal P.$

Using these notions we propose the following projection procedure.
\begin{projProc}[Projection Procedure]\label{proc:ProjProc}
	We consider output $u=(u_1,\ldots,u_S):\Omega \to \mathbb R^s$ of 
	Algorithm~\ref{alg:Algo1} together with its induced directional partitioning $\mathcal I.$
	
	\begin{itemize}
		\item[1.] Compute the partitioning  $\mathcal P(\mathcal I) = \mathcal P_{\mathcal I}$ induced by the directional partitioning $\mathcal I$ as explained above \eqref{eq:FromDirectionalPartitioning2Partitioning}.
		
		\item[2.] Project $u=(u_1,\ldots,u_S):\Omega \to \mathbb R^s$ 
		to $\pi_{\mathcal P_{\mathcal I}}(u)$ using \eqref{eq:DefProjectionPi} for the partitioning  $\mathcal P(\mathcal I) = \mathcal P_{\mathcal I},$ and 
		return $\pi_{\mathcal P_{\mathcal I}}(u)$ as output.	 
	\end{itemize}	
\end{projProc}

We notice that when having a partitioning $\mathcal P_{\mathcal I}$ solving the normal equation in the  space of functions constant on $\mathcal P_{\mathcal I}$ would be an alternative to the above second step which, however, might be more expensive.

\paragraph{A penalty method for the Potts problem based on 
	a majorization-minimization approach for its quadratic penalty relaxation.}

Intuitively, increasing the parameters $\rho$ during the iterations should tie the $u_s$ closer together 
such that the constraint of \eqref{eq:MinimizeMultVar} should be ultimately fulfilled which results in 
an approach for the initial Potts problem \eqref{eq:PottsDiscrete}.  
Recall that
$\rho_{s,s'} = \rho \ c_{s,s'},$ 
was defined by \eqref{eq:rhossprimeIsrhotimesc}, 
where the $c_{s,s'}$ are nonnegative numbers weighting the constraints.  
We here increase $\rho$ while leaving the $c_{s,s'}$ fixed during this process.

\begin{myAlgo}\label{alg:Algo2}
	We consider the Potts problem \eqref{eq:MinimizeMultVar} in $S$ variables (which is equivalent to \eqref{eq:PottsDiscrete} as explained above). We propose the following algorithm 
	for the Potts problem \eqref{eq:MinimizeMultVar}.	
	\begin{itemize}
	    \item[] Let $\rho^{(k)}$ be a strictly increasing sequence 
	    (e.g., $\rho^{(k)} = \tau^k\rho^{(0)},$ with $\rho_0,\tau>1$)
	    and $\delta_k \to 0$ be a strictly decreasing sequence converging to zero 	
	    (e.g., $\delta_k =  \delta_0/\tau^k.$) 	
		Further, let 		
		\begin{align}\label{eq:defTGeneral}
		 t >  2 \sigma_1^{-1/2} S^{-1/2} \|A\| \ \|f\|,
		\end{align}
		where $\sigma_1$ is the smallest non-zero eigenvalue of $C^TC$ with $C$ given by 
		\eqref{eq:CforCoupling}.
		For the particular choice of coupling given by the left-hand and right hand side of \eqref{eq:ChoiceRhos}
		we let 
		\begin{align}\label{eq:LagrangeMultEstSpecial}
		t > \tfrac{2}{S}\|A\| \ \|f\|, \quad 
		\text{ and }\quad 
		 t > 2(2-2\cos(2\pi/S))^{-1/2} S^{-1/2} \|A\| \ \|f\|,
		 \quad \text{ respectively.}
 		\end{align} 
		\item Initialize $u^{(0)}_s := u^{(0,0)}_s$ 
		as discussed in the corresponding paragraph below, (e.g., $u^{(0)}_s = 0$ for all $s.$) 
		
		\item[] Set \ $\rho = \rho^{(0)},\ \rho_{s,s'} = \rho^{(0)}c_{s,s'},\ \delta = \delta_0,\ k,n=0;$ 
		set $L_{\rho}$ according to \eqref{eq:EstBgeneral} (or, in the special cases of \eqref{eq:ChoiceRhos}, as explained below \eqref{eq:EstBgeneral})
		
		\item[A.] While
		\begin{align}\label{eq:TerminationInnerLoopAlg2}
			\left\|u^{(k,n)}_s - u^{(k,n)}_{s'} \right\| > \frac{t}{\rho \sqrt{c_{s,s'}}}, \quad \text{ or } \quad
			\left\| u^{(k,n)}_s - u^{(k,n-1)}_s \right\| > \frac{\delta}{L_{\rho}}							    
		\end{align}
		do
		\begin{align}\notag
		\text{ 1.}\quad & h^{(k,n)}_s = u^{(k,n)}_s + \tfrac{1}{SL_{\rho}^2} A^\ast  f - 
		\tfrac{1}{S L_{\rho}^2} A^\ast A u^{(k,n)}_s - \sum_{s':s' \neq s}\tfrac{\rho_{s,s'}}{L_{\rho}^2}
		(u^{(k,n)}_s-u^{(k,n)}_{s'}),  
		\quad s = 1,\ldots,S,  \\
		\label{eq:backwardStepAlg2}
		\text{ 2.}\quad & \left(u^{(k,n+1)}_1,\ldots,u^{(k,n+1)}_S \right) 
		\in \argmin_{u_1,\ldots,u_S} 	
		\sum _{s=1}^S
		\left[
		\left\| u_s - h^{(k,n)}_s \right\|_2^2 + \tfrac{\gamma \omega_s}{L_{\rho}^2}  \left\| \nabla_{a_s} u_s \, \right\|_0
		\right],  
		\end{align} 
		and  set $n=n+1.$ 
		\item[B.]
		 Set
		 \begin{align}
		 u^{(k+1)}_s = u^{(k+1,0)}_s = u^{(k,n)}_s,
		 \end{align}  
		 set $k = k+1, n=0,$
		 and let $\rho = \rho^{(k)}, \rho_{s,s'} = \rho^{(k)} \ c_{s,s'},  \delta = \delta_k;$
		 set $L_{\rho}$ according to \eqref{eq:EstBgeneral}
		 (or, in the special cases of \eqref{eq:ChoiceRhos}, as below \eqref{eq:EstBgeneral})
		 and  goto A.		
	\end{itemize}

\end{myAlgo}

This approach is inspired by \cite{lu2012iterative} which considers quadratic penalty methods in the sparsity context.
There, the authors are searching for a solution with only a few nonzero entries.
The corresponding prior is separable. 
In contrast to this work, the present work considers a non-separable prior.

\paragraph{Initialization.} 
Although the initialization of Algorithm~\ref{alg:Algo1} and of Algorithm~\ref{alg:Algo2} is not relevant for its convergence properties (cf. Section~\ref{sec:Analysis}), the choice of the initialization influences the final result. (Please note that this also might happen for convex but not strictly convex problems.)
We discuss different initialization strategies.	
The simplest choice is the all-zero initialization
$(u_1^{(0)},...,u_s^{(0)}) = (0,...,0).$
Likewise, one can select
the right hand side of the normal equations of 
the underlying least squares problem, that is $A^Tf$.
A third reasonable choice is the solution of the normal equation itself or an approximation of it.
Using an approximation might in particular be reasonable to get a regularized approximation of the normal equation. A possible strategy to obtain such a regularized initialization is to 
apply a fixed number of Landweber iterations 
\cite{landweber1951iteration} or of the conjugate gradient method  
to the underlying least square problem.
(In our experiments, we initialized Algorithm \ref{alg:Algo1} with the result of 1000 Landweber iterations and Algorithm \ref{alg:Algo2} with $A^Tf$.)

\subsection{A non-iterative algorithm for minimizing the Potts subproblem \eqref{eq:surIt4MultPottsRelaxed}}
\label{sec:ExactSubPotts}

Both proposed algorithms require solving the Potts subproblem \eqref{eq:surIt4MultPottsRelaxed} in the backward step, see 
\eqref{eq:backwardStepAlg1},\eqref{eq:backwardStepAlg2}.
We first observe that \eqref{eq:surIt4MultPottsRelaxed} can be solved 
for each of the $u_s$ separately. 
The corresponding $s$ minimization problems are of the prototypical form
\begin{equation}\label{eq:minDirectionalPotts}
\begin{split}
	\argmin_{u_s:\Omega \to \mathbb R} \| u_s - f\|_2^2 + \gamma'_s \| \nabla_{a_s} u \|_0
\end{split}
\end{equation}
with given data $f$, the jump penalty $\gamma'_s = \tfrac{\gamma \omega_s}{L_{\rho}^2}> 0$ and the direction $a_s\in\Z^2$.
As a next step, we see that \eqref{eq:minDirectionalPotts}
decomposes into univariate Potts problems for data 
along the paths in $f$ induced by $a_s$,
e.g., for $a_s = e_1$ those paths correspond to the rows of $f$ and
we obtain a minimizer $u_s^\ast$ of \eqref{eq:minDirectionalPotts}
by determining each of its rows individually.
The univariate Potts problem
amounts to minimizing
\begin{equation}\label{eq:univariatePottsProblem}
\begin{split}
	P^{\mathrm{id,1d}}_\gamma(x) =  \| x - g\|^2_2 + \gamma \|\nabla x\|_0 \to \min,
\end{split}
\end{equation}
where the data $g$ is given by the restriction of $f$ to the pixels in $\Omega$ of the form $v+a_s z,$ for $z\in\Z$, i.e., $g(z)=f(v+a_s z)$.

Here the offset $v$ is fixed when solving each univariate problem, but varied afterwards to get all lines in the image with direction $a_s.$ 
The target to optimize is denoted by  $x\in\Rn$ and, in the resulting univariate situation,
 $\|\nabla x\|_0= \vert \{ i:x_i \neq x_{i+1} \} \vert$ 
denotes the number of jumps of $x$.
 
It is well-known that the univariate direct problem \eqref{eq:univariatePottsProblem}
has a unique minimizer. Further these particular problems 
can be solved exactly by dynamic programming
\cite{friedrich2008complexity,
	winkler2002smoothers,chambolle1995image,
	mumford1985boundary,mumford1989optimal} which we briefly describe in the
following. For further details we refer to \cite{friedrich2008complexity,storath2018smoothing}.
Assume we have computed minimizers $x^l$ of \eqref{eq:univariatePottsProblem}
for partial data $(g_1,...,g_l)$ for each $l=1,...,r$, $r<n$.
Then the minimum value
of \eqref{eq:univariatePottsProblem} for $(g_1,...,g_{r+1})$
can be found by
\begin{equation}\label{eq:Bellman}
\begin{split}
P^{\mathrm{id,1d}}_\gamma({x^{r+1}} )= \min_{l=1,...,r+1} P^{\mathrm{id,1d}}_\gamma(x^{l-1}) + \gamma
	+\mathcal{E}^{l:r+1},
\end{split}
\end{equation}
where we let $x^0$ be the empty vector, $P^{\mathrm{id,1d}}_\gamma(x^0) = -\gamma$ and
$\mathcal{E}^{l:r+1}$ be the quadratic deviation of
$(g_l,...,g_{r+1})$ from its mean. By denoting the minimizing argument in
\eqref{eq:Bellman} by $l^\ast$ the minimizer $x^{r+1}$ is given by
\begin{equation}
\label{eq:recurrenceRelationX}
x^{r+1} = (x^{l^\ast-1}, \mu_{[l^\ast,r]},...,\mu_{[l^\ast,r]}),
\end{equation}
where $\mu_{[l^\ast,r]}$ is the mean value of $(g_{l^\ast},...,g_r)$.
Thus, we obtain a minimizer for full data $g$ by successively computing
$x^l$ for each $l=1,...,n$.
By precomputing the first and second moments of data $g$ 
and storing only jump locations 
the described method
can be implemented in $\mathcal{O}(n^2)$, \cite{friedrich2008complexity}.
Another way to achieve $\mathcal{O}(n^2)$ is based on the QR decomposition
of the design matrix
by means of Givens rotations, see \cite{storath2018smoothing}.
Furthermore, the search space can be pruned to speed up computations \cite{killick2012optimal,storath2014fast}.

We briefly describe the extensions of the above scheme necessary to 
approach \eqref{eq:univariatePottsProblem} for vector valued-data 
$g\in \R^{n\times C}$ (e.g., the row of a color image). In this situation, the symbol $\mathcal{E}^{l:r+1}$
in \eqref{eq:Bellman} denotes the sum of the quadratic
deviations of $(g_l,\ldots,g_{r+1})$ from its channel-wise means.
Further, $\mu_{[l^\ast,r]}\in\R^{C}$ in \eqref{eq:recurrenceRelationX} 
is the vector of channel-wise means of the data 
$(g_{l^\ast},\ldots,g_r)$. On the computational side, the 
first and second moments of each channel have to be precomputed separately.
It is worth mentioning that the theoretical computational costs 
of the described method grows only linearly in the number of channels \cite{storath2014fast}.
Thus, the proposed algorithm can be efficiently applied to vector-valued images with a high-dimensional codomain.
\section{Analysis}
\label{sec:Analysis}

\subsection{Analytic results}

In the course of the derivation of the proposed algorithms above, we consider the quadratic penalty relaxation \eqref{eq:MinimizeMultVarRelax} of the multivariate Potts problem. Although it is more straight-forward to access algorithmically via our approach, we first note that this problem is still NP hard
(as is the original problem).
\begin{theorem}\label{thm:RelaxedIsNPHard}
	Finding a (global) minimizer of the quadratic penalty relaxation \eqref{eq:MinimizeMultVarRelax} of the multivariate Potts problem is an NP hard problem.
\end{theorem}
The proof is given in Section~\ref{sec:quadPenRelaxRelation} below.\ 
In Section~\ref{sec:ProposedSchemes} we have proposed Algorithm~\ref{alg:Algo1}
to approach the quadratic penalty relaxation of the multivariate Potts problem.
We show that the proposed algorithm converges to a local minimizer and that a feasible point  
of the original multivariate Potts problem is nearby.

\renewcommand{\labelenumi}{\roman{enumi}.}
\begin{theorem}\label{thm:RelaxedConvergesAndTolerance}
	We consider the iterative Potts minimization Algorithm~\ref{alg:Algo1} 
	for the quadratic penalty relaxation \eqref{eq:MinimizeMultVarRelax} of the multivariate Potts problem:
	\begin{enumerate}
	\item	
	Algorithm~\ref{alg:Algo1} computes a local minimizer of  the quadratic penalty relaxation \eqref{eq:MinimizeMultVarRelax} of the multivariate Potts problem for any starting point. The convergence rate is linear.
	\item
	We have the following relation between local minimizers ${\mathcal L}$, global minimizers ${\mathcal G}$ and the fixed points $\mathrm{Fix}(\mathbb{I})$ of the iteration of Algorithm~\ref{alg:Algo1},
	\begin{equation} \label{eq:InclusionRelSets}
	{\mathcal G } \subset \mathrm{Fix}(\mathbb{I}) \subset {\mathcal L}.
	\end{equation}
	\item
	Assume a tolerance $\varepsilon$ we are willing to accept for the distance between the $u_s,$
	i.e.,
	\begin{equation}\label{eq:Tolerance4theUs}
	\sum_{s,s'} c_{s,s'} \|u_s - u_{s'}\|^2_2 = \sum_{s,s'} c_{s,s'} \sum_{i,j}|(u_s)_{ij} - (u_{s'})_{ij}|^2  \leq \varepsilon^2. 	 	 
	\end{equation}
    Running Algorithm~\ref{alg:Algo1} with the choice of the parameter $\rho$ by
    \begin{align}\label{eq:chooseRho}
      \rho > 2 \varepsilon^{-1}  \ \sigma_1^{-1/2} S^{-1/2} \|A\| \|f\|
    \end{align}
    (where $\sigma_1$ is the smallest non-zero eigenvalue of $C^TC$ with $C$ given by 
    \eqref{eq:CforCoupling};
    for the particular choice of the coupling given by \eqref{eq:ChoiceRhos},
    $\sigma_1 = S$ and $\sigma_1 = (2-2\cos(2\pi/S)),$ respectively)
	yields a local minimizer of the quadratic penalty relaxation \eqref{eq:MinimizeMultVarRelax}
	such that the $u_s$ are close up to $\varepsilon,$ i.e., \eqref{eq:Tolerance4theUs} is fulfilled. 	
	\end{enumerate}
\end{theorem}
The proof is given in Section~\ref{sec:EstimatingDistance} below.
A solution of Algorithm~\ref{alg:Algo1} is not a feasible point for the initial Potts problem \eqref{eq:MinimizeMultVar}.
However, we see below that it produces a $\delta$-approximative solution $u^\ast$ in the sense that there is $\mu^*$ and a partitioning $\mathcal P^\ast$ such that 
\begin{align}\label{eq:approxSolution}
  \sum_{s,s'} c_{s,s'} \|u^\ast_s - u^\ast_{s'}\|^2_2 < \delta,
  \qquad \text{ and } \quad   
  L(\mu^\ast) < \delta,
\end{align}
where $L(\mu^\ast)$ is given by \eqref{eq:CaracterizeAsGradientLagrangeMultLmu} below.
In this context note that the conditions for a local minimizer are given by 
$\sum_{s,s'} c_{s,s'} \|u^\ast_s - u^\ast_{s'}\|^2_2 = 0$ and the Lagrange multiplier condition $L(\mu^\ast) = 0.$
So \eqref{eq:approxSolution}
intuitively means that both the constraint and the Lagrange multiplier condition are approximately fulfilled for the partitioning induced by $u^\ast$. 

Further, given a solution of Algorithm~\ref{alg:Algo1} we find a feasible point for the Potts problem \eqref{eq:MinimizeMultVar} (or, equivalently,\eqref{eq:PottsDiscrete}) which is nearby as detailed in the following theorem.

\begin{theorem}\label{thm:EffectRelaxedAlgo4OriginalProblem}
		We consider the iterative Potts minimization Algorithm~\ref{alg:Algo1} for the quadratic penalty relaxation \eqref{eq:MinimizeMultVarRelax} in connection with the (non-relaxed) Potts problem \eqref{eq:MinimizeMultVar}. 
	\begin{enumerate}
		\item Algorithm~\ref{alg:Algo1} produces an approximative solution in the sense of \eqref{eq:approxSolution} of the Potts problem \eqref{eq:MinimizeMultVar}.
		\item The projection procedure (Procedure~\ref{proc:ProjProc}) proposed in  Section~\ref{sec:ProposedSchemes} 
		 applied to the solution $u'=(u'_1,\ldots,u'_S)$ of Algorithm~\ref{alg:Algo1}
		 produces a feasible image $\hat u$ 
		(together with a valid partitioning) for the Potts problem \eqref{eq:MinimizeMultVar}
		which is close to $u'$ in the sense that	
		\begin{align}\label{eq:estDistFeasible}
		\|u_s'-\hat u\| \leq C_1 \varepsilon \qquad \text{for all} \quad s \in \{1,\ldots,S\},
		\end{align}
		where $\varepsilon = \max_{s,s'} \|u'_s-u'_{s'}\|$ quantifies the deviation between the $u_s.$
		Here $C_1 = \# \Omega/4, $ where the symbol $\# \Omega$ denotes the number of elements in $\Omega.$ 	
		If the imaging operator $A$ is lower bounded, 
		i.e., there is a constant $c>0$ such that $\|Au\| \geq c \|u\|$, 
		a local minimizer $u^\ast$ of the Potts problem \eqref{eq:MinimizeMultVar} is nearby, i.e.,
		\begin{align}
		\|u^\ast-\hat u\| \leq \frac{\sqrt \eta} {c} 
		\end{align}
		where
		\begin{equation}\label{eq:DefEta}
		\eta :=  \left( \| A \|^2  \varepsilon C_1^2  
		+ 2 \| A \| C_1 \| f \|_2 \right) \varepsilon.  
		\end{equation}								
	\end{enumerate}	
\end{theorem}
The proof of Theorem~\ref{thm:EffectRelaxedAlgo4OriginalProblem} 
can be found at the end of Section~\ref{sec:MajoMin}, where most relevant statements are already shown  
in Section~\ref{sec:quadPenRelaxRelation}.
Theorem~\ref{thm:EffectRelaxedAlgo4OriginalProblem} theoretically underpins the fact that,
on the application side,
we may use Algorithm~\ref{alg:Algo1} for the Potts problem \eqref{eq:MinimizeMultVar} (accepting some arbitrary small tolerance we may fix in advance).

In addition, in Section~\ref{sec:ProposedSchemes}, we have proposed Algorithm~\ref{alg:Algo2} 
to approach the Potts problem \eqref{eq:MinimizeMultVar}. We first show that 
Algorithm~\ref{alg:Algo2} is well-defined.

\begin{theorem} \label{thm:Algo2isWellDefined}
	Algorithm~\ref{alg:Algo2} is well-defined in the sense that the inner iteration 
	governed by \eqref{eq:TerminationInnerLoopAlg2} terminates,
	i.e., for any $k \in \mathbb N,$ there is $n \in \mathbb N$ such that the termination criterium given by \eqref{eq:TerminationInnerLoopAlg2} holds.	
\end{theorem}
The proof of Theorem~\ref{thm:Algo2isWellDefined} is given in Section~\ref{sec:ConvAlgo2}.
Concerning the convergence properties of Algorithm~\ref{alg:Algo2} we obtain the following results.
\begin{theorem}\label{thm:ConvAlgo2}
	We consider the iterative Potts minimization algorithm (Algorithm~\ref{alg:Algo2})
	for the Potts problem \eqref{eq:MinimizeMultVar}.
	\begin{itemize}
		\item 
		Any cluster point of the sequence $u^{(k)}$ is a local minimizer of the Potts problem \eqref{eq:MinimizeMultVar} 
		(which implicitly implies that the components of each limit $u^\ast$ are equal, i.e., $u_s^{\ast} = u_{s'}^{\ast}$ for all $s,s'.$)
		\item If $A$ is lower bounded, the sequence $u^{(k)}$ produced by Algorithm~\ref{alg:Algo2} has a cluster point
		and the produced cluster points are local minimizers of the Potts problem \eqref{eq:MinimizeMultVar}. 
	\end{itemize}	
\end{theorem}
The proof of Theorem~\ref{thm:ConvAlgo2} can be found in Section~\ref{sec:ConvAlgo2}.

\subsection{Estimates on operator norms and Lagrange multipliers}

\begin{lemma}\label{lem:BbyLcontractive}
	The spectral norm of the block matrix $B$ given by \eqref{eq:DefOfBandg} fulfills
	\begin{equation}\label{eq:EstBgeneralInLemma}
	\| B \|_2 \leq \bigg(\tfrac{1}{S}\|A\|_2^2 + 2 \max_{s \in \{1,\ldots,S \} } \sum_{s': s'\neq s}^{S} \rho_{s,s'}\bigg)^{\frac{1}{2}}.
	\end{equation} 
	
	For the particular choice of constant $\rho_{s,s'} = \rho$ (independent of $s,s'$) as 
	on the left-hand side of 
	\eqref{eq:ChoiceRhos} we have the improved estimate 
	\begin{equation}\label{eq:EstBconstantInLemma}
	\| B \|_2 \leq \bigg(\tfrac{1}{S}\|A\|_2^2 + S \rho\bigg)^{\frac{1}{2}}.
	\end{equation} 
	For only coupling neighboring $u_s$ with the same constant $\rho$, i.e.,  
	the right-hand coupling of \eqref{eq:ChoiceRhos}, we have
	\begin{equation}\label{eq:EstBconstantNeihboringOnlyInLemma}
	\| B \|_2 \leq \bigg(\tfrac{1}{S}\|A\|_2^2 + \alpha \rho\bigg)^{\frac{1}{2}},
	\quad \text{ where } \quad   
	\alpha = \begin{cases}
	4, & \text{ if $S$ is even}, \\
	2 - 2 \cos \left(\frac{\pi(S-1)}{S}\right), & \text{ if $S$ is odd.}
	\end{cases}  
	\end{equation} 
	
\end{lemma}
\begin{proof}
	We decompose the matrix $B$ according to 
	$B =\begin{pmatrix}
	S^{-1/2}\tilde A \\
	\tilde P
	\end{pmatrix}.	
	$ 
	Here $\tilde A$ denotes an $S \times S$- block diagonal matrix with each diagonal entry being equal to 
	$ A,$ where $A$ is the matrix representing the forward/imaging operator; 
	see \eqref{eq:DefOfBandg}. The matrix $\tilde P$ is given as the lower  ${S \choose 2} \times S$- block  
	in \eqref{eq:DefOfBandg} which represents the soft constraints.
	
	Using this decomposition of $B,$
	we may decompose the symmetric and positive (semidefinite) matrix $B^TB$ 
	according to 
	\begin{align}
	B^TB = \tfrac{1}{S} \tilde A^T \tilde A  +  \tilde P^T \tilde P, 
	\end{align}
	where  $\tilde A^T \tilde A$ is an $S \times S$- block diagonal matrix with each diagonal entry being equal to $A^T A,$ and $\tilde P^T \tilde P$ is an  $S \times S$- block diagonal matrix with block entries given by  
	\begin{align}\label{eq:DefRhoTilde}
	\tilde P^T \tilde P  = \begin{pmatrix}		  
	\sum_{k=2}^S\rho_{1,k} I  & -\rho_{1,2} I & -\rho_{1,3} I & \ldots & -\rho_{1,S} I\\
	-\rho_{1,2} I   & \sum_{k=1,k \neq 2}^S\rho_{2,k} I & -\rho_{2,3}I     & \ldots  & -\rho_{2,S} I\\
	& \vdots   &   &   & \vdots    &   \\
	-\rho_{1,S}I   & -\rho_{2,S}I & -\rho_{3,S}I & \ldots & 	\sum_{k=1}^{S-1}\rho_{S,k}I
	\end{pmatrix},
	\end{align}
	with $\rho_{l,k}:= \rho_{k,l}$ for $l>k.$
	Using Gerschgorin's Theorem (see for instance \cite{stoer2013introduction}),
	the eigenvalues of $\tilde P$ are contained in the union of the balls with center 
	$x_r=\sum_{k=1,k \neq r}^S\rho_{r,k}$ and radius $x_r= \sum_{k=1,k \neq r}^S | - \rho_{r,k}|.$  
	These balls are all contained in the larger ball with center $0$
	and radius $2 \cdot \max_r x_r.$ 
	This implies the general estimate \eqref{eq:EstBgeneralInLemma}.
	
	For seeing \eqref{eq:EstBconstantInLemma} we decompose an argument $u=(u_1,\ldots,u_S)$ according to $u= \bar u + u^0$ with
	an ``average'' part $\bar u = (\tfrac{1}{S}\sum_{i=1}^{S}u_i,\ldots,\tfrac{1}{S}\sum_{i=1}^{S}u_i)$ and $u^0:=u-\bar u$ such that $u^0$ has average $0,$ i.e., $\sum_{i=1}^{S}u^0_i = 0,$
	where $0$ denotes the vector containing only zero entries here.
	In the situation of \eqref{eq:EstBconstantInLemma}, 
	the matrix $\tilde P^T \tilde P$ has the form  
	$\tilde P^T \tilde P = \rho(S\cdot I - (1,\ldots,1)(1,\ldots,1)^T)$ 
	We have $\tilde P^T \tilde P \bar u = 0.$
	Further,  $ \tilde P^T \tilde P u^0 = \rho S u^0.$
	Hence, the largest modulus of an eigenvalue of $\tilde P^T \tilde P$ equals $\rho S$
	which in turn shows the estimate \eqref{eq:EstBconstantInLemma}. 
	
	For seeing \eqref{eq:EstBconstantNeihboringOnlyInLemma}, 
	we notice that in case of \eqref{eq:EstBconstantNeihboringOnlyInLemma},
	the matrix $\tilde P^T \tilde P$  has cyclic shift structure with three nonzero entries in each line.
	The discrete Fourier matrix w.r.t.\ the cyclic group of order $S$ diagonalizes  $\tilde P^T \tilde P.$ The corresponding eigenvalues are given by 
	$\lambda_k =  \rho \left(2  - 2 \cos \left ( 2\pi \frac{k}{S} \right)\right),$
	where $k=0,\ldots,S-1.$
	The largest modulus of an eigenvalue is thus given by $4 \ \rho,$ if $S$ is even,
	and by  $\rho \cdot \left( 2 - 2 \cos \left(\frac{\pi(S-1)}{S}\right) \right).$
	\end{proof}

Note that the problem of estimating the operator norm of $B$ in \eqref{eq:EstBgeneralInLemma}
involves computing the operator norm of $\tilde P$ given by \eqref{eq:DefRhoTilde}.
This problem is intimately related to computing the spectral norm of the Laplacian of a corresponding weighted graph (e.g., \cite{spielman2007spectral,chung1997spectral}), in particular, we conclude from this link that the general estimate \eqref{eq:EstBgeneralInLemma} is sharp in the sense that the factor of $2$ in front of the sum cannot be made smaller.
This is because, for a general graph, the spectral radius of the (normalized) Laplacian has spectral norm smaller than two and this factor of two is sharp; cf. \cite{spielman2007spectral,chung1997spectral}.

We recall that we have introduced the concept of a directional partitioning $\mathcal I$ 
and discussed its relation with the concept of a partitioning near
\eqref {eq:FromDirectionalPartitioning2Partitioning} above.
For a function $f: \Omega \to \mathbb R^S$ (representing its $S$ component functions $f_1,\ldots,f_S:\Omega \to \mathbb R$) defined on a grid $\Omega,$ we consider the orthogonal projection $P_\mathcal I$ associated with a directional partition $\mathcal I$ by first sorting the intervals  $\mathcal I$ into 
$\mathcal I_1,\ldots,\mathcal I_S$ according to their associated directions $a_s,$ $s =1,\ldots, S,$
and then letting
\begin{align}\label{eq:DefProjection}
P_\mathcal I f = 
\begin{pmatrix}
P_{\mathcal I_1}    f_1 \\ 
\vdots   \\	
P_{\mathcal I_S}    f_S
\end{pmatrix},
\qquad \text{ where } \quad 
P_{\mathcal I_s}  f_s|_{I} = \frac{\sum_{x \in I} f_s(x) }{ \# I}, 	
\end{align}
i.e., the function $P_{\mathcal I_s}  f_s$ on the interval $I$ is given as the arithmetic mean of $f_i$ on the interval $I$ for all intervals $I \in \mathcal I_s,$ and for all  $s =1,\ldots, S.$ 
Here, the symbol $\# I$ denotes the number of elements in $I.$ 
We note that $P_\mathcal I$ defines an orthogonal projection on the corresponding $\ell^2$ space of discrete functions  $f: \Omega \to \mathbb R^S$ with the norm 
$\|f\|^2 = \sum_{s,i} |(f_s)_i|^2$ where $i$ iterates through all the indices of $f_s.$

We consider a partitioning $\mathcal P$ of $\Omega,$ 
its induced directional partitioning $\mathcal I^{\mathcal P}$ 
w.r.t. a set of $S$ directions $a_1,\ldots,a_S,$ 
and the subspace 
\begin{align}\label{eq:DefSubspacesA}
\mathcal A^{\mathcal P} = P_{\mathcal I^{\mathcal P}}(\ell^2(\Omega, \mathbb R^S))
\end{align}
of functions which are constant on the intervals of the induced directional partitioning 
$\mathcal I^{\mathcal P}$ (which equal the image of the orthogonal projection $P_{\mathcal I^{\mathcal P}}.$)

Functions $g: \Omega \to \mathbb R$ which are piecewise constant w.r.t. a partitioning $\mathcal P,$ 
i.e., they are constant on each  segment $\mathcal P_i$ are  in one-to-one correspondence with the linear subspace    
$\mathcal B^{\mathcal P}$ of $\mathcal A^{\mathcal P}$ given by 
\begin{align}\label{eq:DefSubspacesB}
\mathcal B^{\mathcal P} = \{f \in \mathcal A^{\mathcal P}: f_1 = \ldots = f_S \}
\end{align}
as shown by the following lemma.

\begin{lemma}\label{lem:IdentifyBPandPcPart}
	There is a one-to-one correspondence between the linear space of piecewise constant mappings 
	w.r.t.\ the partitioning $\mathcal P,$ 
	and the subspace $\mathcal B^{\mathcal P}$ of $\mathcal A^{\mathcal P}$ 
	via the mapping $\iota: g \mapsto (g,\ldots,g).$
\end{lemma}
\begin{proof}
	Let $g$ be a piecewise constant mapping	w.r.t.\ the partitioning $\mathcal P,$ 
	then $(g,\ldots,g)$ is constant on each interval $I$ of the induced directional partitioning $\mathcal I^{\mathcal P},$ and $(g,\ldots,g) \in \mathcal B^{\mathcal P}.$ This shows that $\iota$ is well-defined in the sense that its range is contained in $\mathcal B^{\mathcal P}.$ Obviously, $\iota$ is an injective linear mapping so that it remains to show that any $f \in \mathcal B^{\mathcal P}$ is the image under $\iota$ of some $g:\Omega \to \mathbb R$ which is piecewise constant w.r.t.\ the partitioning $\mathcal P.$
	To this end, let  $f \in \mathcal B^{\mathcal P}.$ 
	By definition, $f$ has the form $f = (g,\ldots,g)$ for some $g:\Omega \to \mathbb R.$
	Now, towards a contradiction, assume there is a segment $\mathcal P_i$ and points $x,y \in \mathcal P_i$ with $g(x)\neq g(y).$ Since there is a path $x_0=x,\ldots,x_N=y$ connecting $x,y$ in $\mathcal P_i$ with steps in $a_1,\ldots,a_S,$ we have that for any $i$ 
	there is an interval $I$ in the induced partitioning $\mathcal I^{\mathcal P}$ containing 
	$x_i$ together with $x_{i+1}.$ Since $g$ is constant on each $I$ in  $\mathcal I^{\mathcal P}$
	we get $g(x_i) = g(x_{i+1})$ for all $i$ which implies $g(x)= g(y).$ This contradicts our assumption and shows the lemma. 
\end{proof}

Using the identification given by Lemma~\ref{lem:IdentifyBPandPcPart}, we define, for a given  partitioning $\mathcal P,$ the projection $Q_{\mathcal P}$ onto  $\mathcal B^{\mathcal P}$ by
\begin{align}\label{eq:DefProjectionQ}
Q_{\mathcal P} f  = 
\begin{pmatrix}
\pi_{\mathcal P}  f\\ 
\vdots   \\	
\pi_{\mathcal P}  f
\end{pmatrix},
\qquad \text{ where } \quad 
 \pi  f|_{\mathcal P_i} = \frac{\sum_{s=1}^S\sum_{x \in \mathcal P_i} f_s(x) }{ \# \mathcal P_i \ S}, 	
\end{align}
i.e., we average w.r.t.\ the segment and to all component functions as given by \eqref{eq:DefProjectionPi}. Since the components of 
$Q_{\mathcal P} f$ are all identical, we will not distinguish $Q_{\mathcal P}$ and 
$\pi_{\mathcal P}$ in the following. This means that we also use the symbol $Q_{\mathcal P}f$ 
to denote the scalar-valued function which is piecewise constant on the partitioning $\mathcal P.$

On $\mathcal A^{\mathcal P},$ we consider the problem 
\begin{align}\label{eq:ell2problemOnSegments}
\argmin_{u_1,\ldots,u_S} \sum _{s=1}^S \frac{1}{S}  \left\| Au_s - f \right\|_2^2  
\quad \text{ subject to } \quad Cu=0,
\end{align}
i.e., given the directional partitioning we are searching for a solution which belongs to 
$\mathcal B^{\mathcal P}.$
Here, $C$ denotes the matrix

\begin{equation}\label{eq:CforCoupling}
C  = 
\begin{pmatrix}
c_{1,2} I   & -c_{1,2}I & 0      & \ldots & 0 & 0\\
c_{1,3} I   & 0  & -c_{1,3}I     & \ldots & 0 & 0\\
& \vdots   &   &   & \vdots    &   \\
c_{1,S} I   & 0 & 0 & \ldots & 0 & -c_{1,S}I \\
0   & c_{2,3} I & -c_{2,3}I & \ldots & 0 & 0 \\
&    \vdots &  &   & \vdots    &   \\
0   & c_{2,S} I & 0 & \ldots & 0 & -c_{2,S}I \\ 
&  &  & \vdots & &  \\     
0   & 0 & 0 & \ldots & c_{S-1,S}I & -c_{S-1,S}I \\   
\end{pmatrix},
\end{equation}
where the $c_{s,s'}$ are as in \eqref{eq:MinimizeMultVarRelax}; 
if $c_{s,s'}=0,$ the corresponding line is removed from the constraint matrix $C.$
For the special choices of \eqref{eq:ChoiceRhos}, we have  

\begin{equation}\label{eq:CforTotalCoupling}
C  = 
\begin{pmatrix}
I   & -I & 0      & \ldots & 0 & 0\\
I   & 0  & -I     & \ldots & 0 & 0\\
& \vdots   &   &   & \vdots    &   \\
I   & 0 & 0 & \ldots & 0 & -I \\
0   & I & -I & \ldots & 0 & 0 \\
&    \vdots &  &   & \vdots    &   \\
0   & I & 0 & \ldots & 0 & -I \\ 
&  &  & \vdots & &  \\     
0   & 0 & 0 & \ldots & I & -I \\   
\end{pmatrix},
\quad\text{ and }
\quad
C  = 
\begin{pmatrix}
I   & -I & 0   & 0   & \ldots & 0 & 0& 0\\
0   & I & -I &  0 & \ldots & 0 & 0& 0 \\
0   & 0 & I & -I & \ldots & 0 & 0& 0 \\
&    \vdots & & &   & & \vdots    & &  \\
0   & 0 & 0 & 0& \ldots & I & -I & 0 \\      
0   & 0 & 0 & 0&\ldots & 0 & I & -I\\ 
-I   & 0 & 0 & 0&\ldots & 0 & 0 & I\\     
\end{pmatrix}
\end{equation}
which reflects the constraints $u_1=\ldots=u_S.$
We recall that  
$\mu_{\mathcal P}$ is a Lagrange multiplier of the problem in \eqref{eq:ell2problemOnSegments}
if 
\begin{align}
\min_{u \in \mathcal B^{\mathcal P}} \ \sum _{s=1}^S \frac{1}{S}  \left\| Au_s - f \right\|_2^2 
=   \min_{u \in \mathcal A^{\mathcal P}} \ \sum _{s=1}^S \frac{1}{S}  \left\| Au_s - f \right\|_2^2 + \mu_{\mathcal P}^T C u.
\end{align}
We note that for quadratic problems such as in \eqref{eq:ell2problemOnSegments} Lagrange multipliers always exist \cite{bertsekas2014constrained}. 
We have that 
\begin{align}\label{eq:CaracterizeAsGradientLagrangeMult}
\tfrac{2}{S} P_{\mathcal I^{\mathcal P}}\tilde A^T \tilde A P_{\mathcal I^{\mathcal P}} u_{\mathcal P}^\ast - \tfrac{2}{S} P_{\mathcal I^{\mathcal P}} \tilde A^T \tilde f = 
C^T \mu_{\mathcal P} = P_{\mathcal I^{\mathcal P}} C^T \mu_{\mathcal P}, 
\end{align}
or, in other form, 
\begin{align}\label{eq:CaracterizeAsGradientLagrangeMultLmu}
L(\mu_{\mathcal P}) := \left\| \tfrac{2}{S} P_{\mathcal I^{\mathcal P}}\tilde A^T \tilde A P_{\mathcal I^{\mathcal P}} u_{\mathcal P}^\ast - \tfrac{2}{S} P_{\mathcal I^{\mathcal P}} \tilde A^T \tilde f - 
P_{\mathcal I^{\mathcal P}} C^T \mu_{\mathcal P} \right\| =0, 
\end{align}
where $\tilde A$ is the block diagonal matrix with constant entry $A$ on each diagonal component, 
$\tilde f$ is a block vector of corresponding dimensions with entry $f$ in each component,
and $u_{\mathcal P}^\ast$ is a minimizer of the constraint problem in $\mathcal B^{\mathcal P}$. 
We note that the last equality 
	$C^T \mu_{\mathcal P} = P_{\mathcal I^{\mathcal P}} C^T \mu_{\mathcal P}$ 
	in \eqref{eq:CaracterizeAsGradientLagrangeMult} holds since the left-hand side of \eqref{eq:CaracterizeAsGradientLagrangeMult} is contained 
in the image of $P_{\mathcal I^{\mathcal P}}.$

\begin{lemma}\label{lem:EstLagrangeMultipliers}
	We consider a partitioning $\mathcal P$ of the discrete domain $\Omega,$ 
	and the corresponding problem \eqref{eq:ell2problemOnSegments}. 
	There is a Lagrange multiplier $\mu_{\mathcal P}$ for \eqref{eq:ell2problemOnSegments}
	with 
	\begin{align}\label{eq:LagrangeMultEstGeneral}
	\| \mu_{\mathcal P} \|   \leq 2  \sigma_1^{-1/2} S^{-1/2} \|A\| \|f\|.
	\end{align}
	Here, $\sigma_1$ is the smallest non-zero eigenvalue of $C^TC$ with $C$ given by 
	\eqref{eq:CforCoupling}.
	For the particular choice of $C$ given by the left-hand side of \eqref{eq:CforTotalCoupling}
	we have 
	\begin{align}\label{eq:LagrangeMultEstGeneralFull}
	\| \mu_{\mathcal P} \|   \leq \tfrac{2}{S} \|A\| \|f\|;
	\end{align} 
	and, for the particular choice of $C$ given by the right-hand side of \eqref{eq:CforTotalCoupling}
	we have 
	\begin{align}\label{eq:LagrangeMultEstGeneralConsNeighb}
	\| \mu_{\mathcal P} \|   \leq 2 (2-2\cos(2\pi/S))^{-1/2} S^{-1/2} \|A\| \|f\|,
	\end{align}
	(e.g., for $S=4,$ an eight neigborhood, $ 2-2\cos(2\pi/S) = \sigma_1 = 2.$) 
	In particular, the right-hand side and the constants in all these estimates are independent of the particular partitioning $\mathcal P$. 
\end{lemma}

\begin{proof}
	For any minimizer $u_{\mathcal P}^\ast$ of the constraint problem in $\mathcal B^{\mathcal P},$
	we have that 
	\begin{align}\label{eq:estGradient}	
	\|\tfrac{2}{S} P_{\mathcal I^{\mathcal P}}\tilde A^T \tilde A P_{\mathcal I^{\mathcal P}} u_{\mathcal P}^\ast - \tfrac{2}{S} P_{\mathcal I^{\mathcal P}} \tilde A^T \tilde f\| \leq
	\|\tfrac{2}{S}\tilde A^T \tilde A u_{\mathcal P}^\ast - \tfrac{2}{S} \tilde A^T \tilde f \|
	\leq  \tfrac{2}{S} \|A\| \|\tilde f\|   \leq \tfrac{2 \sqrt S}{S} \|A\| \|f\|,
	\end{align}
	where we recall that
	$\tilde A$ is the block diagonal matrix with constant entry $A,$  
	and $\tilde f$ is a block vector with entry $f$ in each component.
	The first inequality is a consequence of the fact that $P_{\mathcal I^{\mathcal P}}$ is an orthogonal projection.
	The second inequality may be seen by evaluating the term for the constant zero function
	(which always belongs to $\mathcal B^{\mathcal P}$ ) as a candidate 
	and by noting that $\|A^T \| =\|A\|.$ 
	
	Using \eqref{eq:CaracterizeAsGradientLagrangeMult}, we have 
	$\| C^T \mu_{\mathcal P}\| \leq \tfrac{2}{\sqrt S} \|A\| \|f\|.$
	Choosing $\mu_{\mathcal P}$ in the complement of the zero space of $C^T,$ 
	we get 
	\begin{align}\label{eq:EstEigenvaluesOfC}
	\| C^T \mu_{\mathcal P}\| \geq   \inf_{x \in \left(\ker{C^T}\right)^\perp,\|x\|=1} \|C^Tx\| \ 	\|\mu_{\mathcal P}\|.
	\end{align}
	We observe that finding the infimum in \eqref{eq:EstEigenvaluesOfC} corresponds to finding the square root of the smallest nonzero eigenvalue of $C^TC.$ This is because (i) the nonzero eigenvalues of $C^TC$
	equal the nonzero eigenvalues of $CC^T,$ i.e.,
	\begin{align} \label{eq:Sigma1OfGraphLaplacian}
	\min\left\{\sigma: \sigma \in \mathrm{spectrum}(CC^T)\setminus \{0\} \right\} =
	\min\left\{\sigma: \sigma \in \mathrm{spectrum}(C^TC)\setminus \{0\} \right\} = \sigma_1,
	\end{align}
	where $\sigma_1$ is the smallest non-zero eigenvalue of $C^TC$.
	Further, (ii) for $x \in \left(\ker{C^T}\right)^\perp,$
	$\|C^Tx\|^2 = \langle  x, CC^Tx \rangle \geq 
	\min\left\{\sigma: \sigma \in \mathrm{spectrum}(CC^T)\setminus \{0\} \right\} \|x\|^2.$
	Hence, using \eqref{eq:Sigma1OfGraphLaplacian} in \eqref{eq:EstEigenvaluesOfC} we get  that
	$\| C^T \mu_{\mathcal P}\| \geq  \sqrt{\sigma_1} 	\|\mu_{\mathcal P}\|,$
	and together with 
	\eqref{eq:CaracterizeAsGradientLagrangeMult} and \eqref{eq:estGradient}, we obtain
	\begin{align}
		\|\mu_{\mathcal P}\| \leq \sigma^{-1/2} \| C^T \mu_{\mathcal P}\| \leq 2 \sigma^{-1/2} S^{-1/2} \|A\| \|f\|
	\end{align}
	which shows \eqref{eq:LagrangeMultEstGeneral}.	
	
	Now we consider the particular choice of $C$ given by the left-hand side of \eqref{eq:CforTotalCoupling}.
	Similar to the derivation in \eqref{eq:DefRhoTilde}, we have that 
	$C^TC= S \cdot I - (1,\ldots,1)(1,\ldots,1)^T).$ Further, the constants constitute the kernel of $C^TC$
	and any vector $u$ in its orthogonal complement is mapped to $S u.$  
	Hence, $\sigma_1 = S$ which shows \eqref{eq:LagrangeMultEstGeneralFull}.
	
	Finally, we consider the particular choice of $C$ given by the right-hand side of \eqref{eq:CforTotalCoupling}.
	As already explained in the proof of 
	Lemma~\ref{lem:BbyLcontractive}
	the discrete Fourier transform shows that the corresponding eigenvalues are given by 
	$\lambda_k =  \rho \left(2  - 2 \cos \left ( 2\pi \frac{k}{S} \right)\right),$
	where $k=0,\ldots,S-1.$ The smallest nonzero eigenvalue is thus given by $2-2\cos(2\pi/S).$	
	This shows \eqref{eq:LagrangeMultEstGeneralConsNeighb} which completes the proof of the lemma.  	
\end{proof}

\subsection{The quadratic penalty relaxation of the Potts problem 
	and its relation to the Potts problem}\label{sec:quadPenRelaxRelation}

In this subsection we reveal some relations between 
the Potts problem and its quadratic penalty relaxation;
in particular, we show Theorem~\ref{thm:RelaxedIsNPHard} and parts of
Theorem~\ref{thm:EffectRelaxedAlgo4OriginalProblem}.
We start out to show that the quadratic penalty relaxation of the Potts problem is NP hard which was formulated as
Theorem~\ref{thm:RelaxedIsNPHard}.
\begin{proof}[Proof of Theorem~\ref{thm:RelaxedIsNPHard}]
	We consider the quadratic penalty relaxation \eqref{eq:MinimizeMultVarRelax} of the multivariate Potts problem in its equivalent form \eqref{eq:RelaxedProblemInShortNotation} which reads	
	\begin{equation*}
	P_{\gamma,\rho}(u_1,\ldots,u_S) = \left\| B(u_1,\ldots,u_S)^T - g \right\|_2^2 + \gamma \ \Big\| \ D(u_1,\ldots,u_S) \ \Big\|_{0,\omega}.
	\end{equation*}
	with
	$B$ and $g$ given by \eqref{eq:DefOfBandg} and $D$ given by \eqref{eq:DefDiffOp}.
	We serialize $ u:(u_1,\ldots,u_S): \Omega \to \mathbb R^S$ into a  function 
	$\hat u:  X \to \mathbb R$ with $X \subset \mathbb Z$ being a discrete interval of size 
	$S \#\Omega$	as follows:
	For $u_s,$ we consider the discrete lines in the image with direction $a_s$ and interpret $u$
	on these lines as a vector; then we concatenate these vectors starting with the one corresponding to the leftmost upper line to obtain a vector of length $\#\Omega;$
	for each $s,$ we obtain such a vector and we again concatenate these vectors starting with index $s=1,2,\ldots$ to obtain the resulting object which we denote by $\hat u.$
	Using this serialization we may arrange $B,g$ and $D$ accordingly to obtain
	the univariate Potts problem
	\begin{equation*}
	\hat P_{\gamma,\rho}(\hat u) = \left\| \hat B \hat u - \hat g \right\|_2^2 + 
	\gamma \ \Big\| \hat \omega  \nabla \hat u \Big\|_{0},
	\quad \text{ where } 
	 \hat \omega: X \to [0,\infty)
	\end{equation*}
	is a weight vector, $\omega  \nabla \hat u$ denotes pointwise multiplication, and $\hat B,\hat g$ are the matrix and the vector corresponding to  $B,g$ 
	w.r.t.\ the serialization.
	The weight vector may be zero which in particular happens at the line breaks, i.e., those indices where two vectors have been concatenated in the above procedure.
	More precisely, constant data induce a directional segmentation on $\Omega$ and the image of the directional segmentation under the above serialization procedure induces a partitioning of the unvariate domain $D;$ precisely between these segments, the weight vectors equals zero.
	Now, for each segment $[d_1,\ldots,d_r]$ in $D,$ we transform 
	the basis $\delta_{d_1},\ldots,\delta_{d_r}$ to the basis 
	$\delta_{d_2}-\delta_{d_1}, \ldots, \delta_{d_r}-\delta_{d_{r-1}}, \tfrac{1}{r}\sum_{l=1}^r \delta_{d_l} $ obtained by neighboring differences and the average.
	As a result (which is in detail elaborated in \cite{storath2013jump}), we obtain a problem of the form
	\begin{equation}\label{eq:rewriteToSparsity}
	\hat P_{\gamma,\rho}(\hat u) = \left\| \tilde B \tilde u - \tilde b \right\|_2^2 + 
	\gamma \ \Big\| \hat \omega   \tilde u \Big\|_{0},
	\quad \text{ where } 
	\hat \omega: D \to [0,\infty)
	\end{equation}
	which is a sparsity problem and which is known to be NP hard; see, for instance, \cite{storath2013jump}.
	This shows the assertion.
	\end{proof}

We next characterize the local minimizers of the relaxed Potts problem \eqref{eq:MinimizeMultVarRelax} 
and of the Potts problem \eqref{eq:PottsDiscrete}.

\begin{lemma}\label{lem:CharLocMinQuPenRelax}
A local minimizer $u=(u_1,\ldots,u_S)$ of the quadratic penalty relaxation \eqref{eq:MinimizeMultVarRelax} is characterized as follows: let $\mathcal I$ be the directional partitioning induced by the minimizer $u,$ and $\mathcal P = \mathcal P_{\mathcal I}$ be the induced partitioning, then $u$ is a minimizer of  the problem
\begin{align}\label{eq:QuadraticRelaxOnAp}
\min_{u \in \mathcal A^{\mathcal P}}   F_{\rho}(u),
\quad \text{ where } \quad
F_{\rho}(u) = \sum\nolimits_{s=1}^S \tfrac{1}{S}  \left\| Au_s - f \right\|_2^2 + 
\rho\|C u \|^2.
\end{align}
Conversely, if $u$ minimizes \eqref{eq:QuadraticRelaxOnAp} on $\mathcal A^{\mathcal P},$  then $u$ is a minimizer of 
the relaxed Potts problem \eqref{eq:MinimizeMultVarRelax}.	
\end{lemma}
\begin{proof}
Let $u=(u_1,\ldots,u_S)$ be a local minimizer of the quadratic penalty relaxation \eqref{eq:MinimizeMultVarRelax}. Hence, there is a neigborhood $\mathcal U$ of $u$ such that, 
for any $v \in \mathcal U,$ $P_{\gamma, \rho}(v) \geq P_{\gamma, \rho}(u).$
Now if $v \in \mathcal A^{\mathcal P}$ and $\|v-u\|$ is small, then
$\sum _{s=1}^S \omega_s  \left\| \nabla_{a_s} u_s \, \right\|_0$
$ =\sum _{s=1}^S \omega_s  \left\| \nabla_{a_s} v_s \, \right\|_0$
which implies that 
\begin{align}
	    F_{\rho}(u) = 
	    P_{\gamma, \rho}(u) - \gamma \sum _{s=1}^S \omega_s  \left\| \nabla_{a_s} u_s \,\right\|_0 
	    \leq P_{\gamma, \rho}(v) - \gamma \sum _{s=1}^S \omega_s  \left\| \nabla_{a_s} v_s \, \right\|_0 
	      =  F_{\rho}(v). 
\end{align}
This shows that $u$ minimizes \eqref{eq:QuadraticRelaxOnAp}. 
Conversely, we assume that $u$ minimizes \eqref{eq:QuadraticRelaxOnAp}.
If the directional partitioning $\mathcal I'$ induced by $u$ is coarser than $\mathcal I$
consider the coarser directional partitioning $\mathcal I'$ instead of $\mathcal I.$
Let the maximum norm of $h=(h_1\ldots,h_S)$ be smaller than the height of the smallest jump of $u,$
then, for $u+h,$  
\begin{equation}\label{eq:EstJumps4MinCharacterization}
\sum _{s=1}^S \omega_s  \left\| \nabla_{a_s} (u_s+h_s) \, \right\|_0 \geq 
\sum _{s=1}^S \omega_s  \left\| \nabla_{a_s} u_s \, \right\|_0.  	
\end{equation}
If inequality holds in \eqref{eq:EstJumps4MinCharacterization}, the continuity of $F_{\rho}$ 
implies that $F_{\rho}(u+h) \geq  F_{\rho}(u) - \varepsilon$ for small enough $h$ and arbitrary $\varepsilon.$ Hence, 
\begin{align} \notag
	P_{\gamma, \rho}(u) &= F_{\rho}(u) + \gamma \sum _{s=1}^S \omega_s  \left\| \nabla_{a_s} u_s \, \right\|_0
	\leq F_{\rho}(u+h) - \gamma \min_s \omega_s + \gamma \sum _{s=1}^S \omega_s  \left\| \nabla_{a_s} (u_s+h_s) \, \right\|_0 + \varepsilon \\
	&\leq F_{\rho}(u+h) + \gamma \sum _{s=1}^S \omega_s  \left\| \nabla_{a_s} (u_s+h_s) \, \right\|_0
	= P_{\gamma, \rho}(u+h),
\end{align}
if we choose $\varepsilon$ small enough. If equality holds in \eqref{eq:EstJumps4MinCharacterization}
we have that $u+h \in \mathcal A^{\mathcal P}$ which implies $F_{\rho}(u) \leq F_{\rho}(u+h)$ since $u$
is a minimizer of $F_{\rho}$ on $\mathcal A^{\mathcal P}.$ This in turn implies 
$P_{\gamma, \rho}(u) \leq P_{\gamma, \rho}(u+h)$ by the assumed equality in \eqref{eq:EstJumps4MinCharacterization}. Together, in any case, $P_{\gamma, \rho}(u) \leq P_{\gamma, \rho}(u+h)$ for any small perturbation $h.$ This shows that $u$ is a local minimizer of 
$P_{\gamma, \rho}$ which completes the proof.
\end{proof}

\begin{lemma}\label{lem:CharLokMinPotts}
	We consider a function $u^\ast:\Omega \to \mathbb R$ and its induced partitioning $\mathcal P.$
	Then $u$ is a local minimizer of the Potts problem \eqref{eq:PottsDiscrete}, if and only if
	$(u^\ast,\ldots,u^\ast)$ minimizes \eqref{eq:ell2problemOnSegments} w.r.t.\ 
	$\mathcal P.$  
\end{lemma}

\begin{proof} 
	Since the proof of this statement is very similar to the proof of Lemma~\ref{lem:CharLocMinQuPenRelax} we keep it rather short and refer to the proof of Lemma~\ref{lem:CharLocMinQuPenRelax} if more explanation is necessary.
	Let $u$ be a minimizer of \eqref{eq:PottsDiscrete} which is equivalent to 
	$\bar u = (u,\ldots,u)$ being a minimizer of \eqref{eq:PottsDiscreteSvariables}.
	There is a neigborhood $\mathcal U$ of $\bar u$ such that, 
	for any $\bar v=(v,\ldots,v) \in \mathcal U,$ $P_{\gamma}(v) \geq P_{\gamma}(u).$
	For $\bar v \in \mathcal B^{\mathcal P}$ with small $\|\bar v- \bar u\|,$ we have
	$\sum _{s=1}^S \omega_s  \left\| \nabla_{a_s} u \, \right\|_0$
	$ =\sum _{s=1}^S \omega_s  \left\| \nabla_{a_s} v \, \right\|_0.$
	Hence, by the definition of $P_\gamma$ in \eqref{eq:PottsDiscreteSvariables} 
	$	\| Au - f \|_2^2  \leq  \| Av - f \|_2^2 $ 
	which shows that $(u^\ast,\ldots,u^\ast)$ minimizes \eqref{eq:ell2problemOnSegments}.
	
	Conversely, let $\bar u = (u,\ldots,u)$ be a minimizer of \eqref{eq:ell2problemOnSegments} 
	with the partitioning $\mathcal P$ induced by $u$.
	For $\bar h=(h,\ldots,h)$ with absolute value smaller than the minimal height of a jump of $u,$
	we have the estimate
	$ 
	\sum _{s=1}^S \omega_s  \left\| \nabla_{a_s} (u+h) \, \right\|_0 
	$
	$
	\geq 
	\sum _{s=1}^S \omega_s  \left\| \nabla_{a_s} u \, \right\|_0.  	
	$
	If inequality holds in this estimate, the continuity of $F_{\rho}$ 
	implies that 
	$	\| A(u+h) - f \|_2^2  \geq  \| Au - f \|_2^2 - \varepsilon$
	for small enough $h$ and arbitrary $\varepsilon.$ Hence, 
	$
	P_{\gamma}( \bar u) \leq \| A(u+h) - f \|_2^2 - \gamma \min_s \omega_s + \gamma \sum _{s=1}^S \omega_s  \left\| \nabla_{a_s} (u+h) \, \right\|_0 + \varepsilon$ 
	$
	\leq  P_{\gamma, \rho}(\bar u+\bar h)
	$
	if  $\varepsilon$ is small. If equality holds above, i.e, 
	$ 
	\sum _{s=1}^S \omega_s  \left\| \nabla_{a_s} (u+h) \, \right\|_0 
	$
	$
	= 
	\sum _{s=1}^S \omega_s  \left\| \nabla_{a_s} u \, \right\|_0,  	
	$
	then $\bar u+ \bar h \in \mathcal B^{\mathcal P}$ which implies that 
	$\| Au - f \|_2^2  \leq \| A(u+h) - f \|_2^2$
	since $\bar u$ is a minimizer of the corresponding function on $\mathcal B^{\mathcal P}.$ 
	As a consequence
	$P_{\gamma}(\bar u) \leq P_{\gamma}(\bar u+ \bar h)$ for any small perturbation $h.$ This shows that $u$ is a local minimizer of 
	$P_{\gamma}$ which completes the proof.
\end{proof}

\begin{proposition}\label{prop:aproxLocMin}
	Any local minimizer of the quadratic penalty relaxation \eqref{eq:MinimizeMultVarRelax} is an approximate local minimizer in the sense of \eqref{eq:approxSolution} of the Potts problem \eqref{eq:MinimizeMultVar}.
\end{proposition}
\begin{proof}
By Lemma~\ref{lem:CharLocMinQuPenRelax},  a local minimizer $u=(u_1,\ldots,u_S)$ of the quadratic penalty relaxation \eqref{eq:MinimizeMultVarRelax} is a minimizer of  the problem
\eqref{eq:QuadraticRelaxOnAp}.
Let us thus consider a local minimizer $u$ of \eqref{eq:MinimizeMultVarRelax} with induced partitioning 
 $\mathcal P = \mathcal P_{\mathcal I}.$ Since $u$ minimizes \eqref{eq:QuadraticRelaxOnAp}, we have 
\begin{align}\label{eq:CaracterizeAsGradientLagrangeMultloc}
\tfrac{1}{S} P_{\mathcal I}\tilde A^T \tilde A P_{\mathcal I} u - \tfrac{1}{S} P_{\mathcal I} \tilde A^T \tilde f + \rho P_{\mathcal I} C^T C P_{\mathcal I} u = 0 
\end{align} 
since the gradient projected to $\mathcal A^{\mathcal P}$  equals zero for any local minimizer  
of the restricted problem on the subspace $\mathcal A^{\mathcal P}.$ 
(The notation is chosen as in \eqref{eq:CaracterizeAsGradientLagrangeMultLmu} above.) We define 
$\mu$ by 
$\mu = \rho  C P_{\mathcal I} u$ and obtain 
\begin{equation}\label{eq:MuIsZero}
	L(\mu) = \|\tfrac{1}{S} P_{\mathcal I}\tilde A^T \tilde A P_{\mathcal I} u - \tfrac{1}{S} P_{\mathcal I} \tilde A^T \tilde f + P_{\mathcal I} C^T \mu \| = 0
\end{equation}
 by \eqref{eq:CaracterizeAsGradientLagrangeMultloc}.
It remains to show that $\|Cu\|$ becomes small.
To this end, we observe that, by Lemma~\ref{lem:estCu}, 
for arbitrary $v=(v_1,\ldots,v_S)\in \mathcal A^{\mathcal P}$, 
$  
\| C  v \| =  \| C P_{\mathcal I} v \| \leq \tfrac{1}{\rho}\|\mu^\ast\|+ \sqrt{\tfrac{F_{\rho}(v)- \min_{x \in \mathcal A^{\mathcal P}} F_{\rho}(x)}{\rho}},
$
where $\mu^\ast$ is an arbitrary Lagrange multiplier of \eqref{eq:ell2problemOnSegments}.
Plugging in the minimizer $u$ for $v$ yields $\| C  u \| < \tfrac{1}{\rho}\|\mu^\ast\|.$
Thus letting  $\delta = \tfrac{1}{\rho}\|\mu^\ast\|,$ we have 
\begin{align}
\sum_{s,s'} c_{s,s'} \|u^\ast_s - u^\ast_{s'}\|^2_2 = \|Cu\|^2 < \delta,
\end{align}
and  $L(\mu) = 0$ by \eqref{eq:MuIsZero} which by \eqref{eq:approxSolution} shows the assertion and completes the proof.
\end{proof}
For the proof of Proposition~\ref{prop:aproxLocMin} as well as in the following we need the next lemma. Similar statements are \cite[Proposition 13]{lan2013iteration} and  \cite[Lemma 2.5]{lu2012iterative}.
However, since there are differences concerning the precise estimate in these references, and the setup here is slightly different, we provide a brief proof here for the readers convenience. 
\begin{lemma}\label{lem:estCu}
	Let $\mathcal P$ be a partitioning and $\mathcal I = \mathcal I_{\mathcal P}$ be the corresponding induced partitioning. 
	For arbitrary $v=(v_1,\ldots,v_S)\in \mathcal A^{\mathcal P}$, 
	\begin{align}  
	\| C  v \| =  \| C P_{\mathcal I} v \| \leq \tfrac{1}{\rho}\|\mu^\ast\|+ \sqrt{\tfrac{F_{\rho}(v)- \min_{x \in \mathcal A^{\mathcal P}} F_{\rho}(x)}{\rho}},
	\end{align}
	where $\mu^\ast$ is an arbitrary Lagrange multiplier of \eqref{eq:ell2problemOnSegments}.
\end{lemma}

\begin{proof}
By \cite[Corollary 2]{lan2013iteration}, we have for arbitrary $v=(v_1,\ldots,v_S)\in \mathcal A^{\mathcal P}$ that
\begin{equation}\label{eq:EstfromLan}
\sum\nolimits_{s=1}^S \tfrac{1}{S}  \left\| Av_s - f \right\|_2^2  
 - \min_{(y,\ldots,y) \in \mathcal B^{\mathcal P}} \left\| Ay - f \right\|_2^2
   \geq - \|\mu^\ast\|   \ \|Cv\|.
\end{equation}
Then,
\begin{align}\label{eq:est4quad}
F_{\rho}(v)- \min_{x \in \mathcal A^{\mathcal P}} F_{\rho}(x) 
&\geq \sum\nolimits_{s=1}^S \tfrac{1}{S}  \left\| Av_s - f \right\|_2^2 + 
\rho\|C v \|^2 - \min_{(y,\ldots,y) \in \mathcal B^{\mathcal P}} F_{\rho}(y,\ldots,y) \notag \\
&= \sum\nolimits_{s=1}^S \tfrac{1}{S}  \left\| Av_s - f \right\|_2^2 + 
\rho\|C v \|^2 - \min_{(y,\ldots,y) \in \mathcal B^{\mathcal P}} \left\| Ay - f \right\|_2^2  \notag \\
& \geq  \rho\|C v \|^2    - \|\mu^\ast\|   \ \|Cv\|.
\end{align}
For the first inequality we wrote down the definition of $F_{\rho}$ and restricted the set with respect to which the minimum is formed which results in a potentially larger function value.  
For the second inequality we notice that, for $(y,\ldots,y) \in \mathcal B^{\mathcal P}$, we have $C(y,\ldots,y)=0,$ and for the last inequality we employed \eqref{eq:EstfromLan}. 
Now, writing 
$ z^2  - \frac{\|\mu^\ast\|}{\|\rho\|} z$  
$=  z^2  - \frac{\|\mu^\ast\|}{\rho} z + \left(\tfrac{\|\mu^\ast\|}{2\rho}\right)^2  - \left(\tfrac{\|\mu^\ast\|}{2\rho}\right)^2$
$= (z - \tfrac{\|\mu^\ast\|}{2\rho} )^2 - \left(\tfrac{\|\mu^\ast\|}{2\rho}\right)^2 $ 
and plugging this into \eqref{eq:est4quad} with $z := \|Cv\|$ yields
\begin{align}\label{eq:est4quad2}
\tfrac{F_{\rho}(v)- \min_{x \in \mathcal A^{\mathcal P}} F_{\rho}(x)}{\rho} 
\geq (\|Cv\| - \tfrac{\|\mu^\ast\|}{2\rho} )^2 - \left(\tfrac{\|\mu^\ast\|}{2\rho}\right)^2, 
\end{align}
and hence 
\begin{align}\label{eq:est4quad3}
\left|\|Cv\| - \tfrac{\|\mu^\ast\|}{2\rho} \right| \leq \sqrt{\tfrac{F_{\rho}(v)- \min_{x \in \mathcal A^{\mathcal P}} F_{\rho}(x)}{\rho} + \left(\tfrac{\|\mu^\ast\|}{2\rho}\right)^2 } 
\leq \sqrt{\tfrac{F_{\rho}(v)- \min_{x \in \mathcal A^{\mathcal P}} F_{\rho}(x)}{\rho}} + \tfrac{\|\mu^\ast\|}{2\rho}
\end{align}
where the last inequality is a consequence of the fact that the unit ball w.r.t.\ the $\ell^1$ norm 
is contained in the the unit ball w.r.t.\ the $\ell^2$ norm.
As a consequence,
$
\|Cv\| \leq \sqrt{\tfrac{F_{\rho}(v)- \min_{x \in \mathcal A^{\mathcal P}} F_{\rho}(x)}{\rho}} + \tfrac{\|\mu^\ast\|}{2\rho} + \tfrac{\|\mu^\ast\|}{2\rho}
$
which completes the proof. 
\end{proof}

Next, we see that for any local minimizer of the quadratic penalty relaxation \eqref{eq:MinimizeMultVarRelax}, we can find a nearby feasible point using  
the projection procedure (Procedure~\ref{proc:ProjProc}) proposed in  Section~\ref{sec:ProposedSchemes}. 
Further, if the imaging operator $A$ is lower bounded, we find a nearby minimizer.

\begin{proposition}\label{pro:theoretResOnProc1}
	Procedure~\ref{proc:ProjProc} applied to a local minimizer $u'=(u'_1,\ldots,u'_S)$ of 
	the quadratic penalty relaxation \eqref{eq:MinimizeMultVarRelax} produces a feasible image $\hat u$ (together with a valid partitioning) for the Potts problem \eqref{eq:MinimizeMultVar}
	which is close to $u'$ in the sense that
	\begin{align}\label{eq:estDistProFeasible}
	\|u_s'-\hat u\| \leq C_1 \varepsilon \qquad \text{for all} \quad s \in \{1,\ldots,S\},
	\end{align}
	where $\varepsilon = \max_{s,s'} \|u'_s-u'_{s'}\|$ quantifies the deviation between the $u_s.$
	Here $C_1 = \# \Omega/4, $ where the symbol $\# \Omega$ denotes the number of elements in $\Omega.$

	If the imaging operator $A$ is lower bounded, 
	i.e., there is a constant $c>0$ such that $\|Au\| \geq c \|u\|$, 
	a local minimizer $u^\ast$ of the Potts problem \eqref{eq:MinimizeMultVar} is nearby, i.e.,
	\begin{align}
	\|u^\ast-\hat u\| \leq \frac{\sqrt \eta} {c} 
	\end{align}
	where
	\begin{equation}\label{eq:DefEtaPro}
	\eta :=  \left( \| A \|^2  \varepsilon C_1^2  
	+ 2 \| A \| C_1 \| f \|_2 \right) \varepsilon.  
	\end{equation}			
\end{proposition}	

\begin{proof}
	We denote the directional partitioning induced by $u'$ by $\mathcal I$ 
	and the corresponding induced partitioning by $\mathcal P = \mathcal P_{\mathcal I}.$
	We note that Procedure~\ref{proc:ProjProc} applied to $u'$ precisely produces 
	\begin{align}
		(\hat u,\ldots,\hat u)= Q_\mathcal P u',		
	\end{align}
	with the projection $Q_\mathcal P$ given by \eqref{eq:DefProjectionQ}.
	We first note, that the average $(\bar u)_{ij} = \tfrac{1}{S} \sum_{s=1}^{S} (u_s')_{ij}$
	fulfills $|(\bar u)_{ij}- (u_s')_{ij}|<\varepsilon.$ 
	Further, the function value of $\hat u$ which is piecewise constant w.r.t.\ $\mathcal P$ is obtained by $\hat u|_{\mathcal P_i} = \sum_{x \in \mathcal P_i}\bar u(x)/ \# \mathcal P_i.$
	Hence, we may estimate
	\begin{align}
	 \|u_s'-\hat u\|_2^2 \leq \varepsilon   L,
	\end{align}
	where $L$ is the maximal length of a path connecting any two pixels as given  
	by Definition~\ref{def:EquivPath}.
	As a worst case estimate, we get $L\leq C_1$ where we define $C_1$ as one fourth of the number of elements in $\Omega,$ i.e., $C_1 = \tfrac{\# \Omega}{4}.$
	This shows \eqref{eq:estDistProFeasible}.
	
	For $F_{\rho}$ given by \eqref{eq:QuadraticRelaxOnAp}, we have 	 
	\begin{align}\label{eq:estFucValueFrho}
	F_{\rho}(u') &\leq   F_{\rho}(\hat u,\ldots,\hat u) 
	= \sum\nolimits_{s=1}^S \tfrac{1}{S}  \left\| A\hat u - f \right\|_2^2 \notag\\ 
    &\leq \sum\nolimits_{s=1}^S \tfrac{1}{S}  
     \left(\left\| A\hat u  - A u_s'  \right\|_2  +  \left\| Au_s' - f \right\|_2\right)^2 \notag \\
	&\leq \sum\nolimits_{s=1}^S \tfrac{1}{S}  
	\left(\| A \|  \varepsilon C_1  +  \left\| Au_s' - f \right\|_2\right)^2  \\
	&\leq \| A \|^2  \varepsilon^2 C_1^2  
	+ 2 \| A \|  \varepsilon C_1 \sum\nolimits_{s=1}^S \tfrac{1}{S}  
	\left\| Au_s' - f \right\|_2 
	+  \sum\nolimits_{s=1}^S \tfrac{1}{S}  
	  \left\| Au_s' - f \right\|_2^2  \notag\\
	&\leq \eta + F_{\rho}(u'),  \notag
	\end{align}
	with
	\begin{equation}
		\eta =  \left( \| A \|^2  \varepsilon C_1^2  
		+ 2 \| A \| C_1 \| f \|_2 \right) \varepsilon, 
	\end{equation}		
	as given in \eqref{eq:DefEtaPro}.
	The first inequality holds since as a local minimizer of the quadratic penalty relaxation \eqref{eq:MinimizeMultVarRelax}, 
	$u'$ is the global minimizer of $F_{\rho}$ on $\mathcal A^{\mathcal P}$
	by Lemma~\ref{lem:CharLocMinQuPenRelax} and since $(\hat u,\ldots,\hat u) \in \mathcal A^{\mathcal P}$ by construction. The next inequalities apply the triangle inequality and estimates on matrix norms. The last inequality is a consequence of the fact that 
	$\sum\nolimits_{s=1}^S \tfrac{1}{S}  
	\left\| Au_s' - f \right\|_2 \leq \|f\|_2.$ Otherwise, if $\| Au_s' - f \|_2 > \|f \|_2,$ choosing $u_s'=0$  would yield a lower function value which would contradict the minimality of $u'.$
	
	Now consider the partitioning $\mathcal P'$ induced by $\hat u,$ and the corresponding minimizer 
	$u^\ast,$ i.e.,
	\begin{align}
	    (u^\ast,\ldots,u^\ast)  = \argmin_{u \in \mathcal B^{\mathcal P'}} F_{\rho}(u)
	\end{align}
	where, for $(u,\ldots,u) \in \mathcal B^{\mathcal P'},$ we have 
	$F_{\rho}(u,\ldots,u) =  \left\| Au - f \right\|_2^2.$
	By Lemma~\ref{lem:CharLokMinPotts}, $u^\ast$ is a local minimizer of the Potts problem \eqref{eq:PottsDiscrete}.
	On the other hand, by orthogonality in an inner product space, we have   
	\begin{equation}
		Au^\ast  =  P_{A \left(\mathcal B^{\mathcal P'}\right)} f, \quad\text{ and }\quad 
		\|f- P_{A \left(\mathcal B^{\mathcal P'}\right)} f\|^2 = \min_{u \in \mathcal B^{\mathcal P'}} F_{\rho}(u),
	\end{equation}
	where  $P_{A \left(\mathcal B^{\mathcal P'}\right)}$ denotes the orthogonal projection onto the image of 
	$\mathcal B^{\mathcal P'}$ under the linear mapping $A.$
	Thus,
	\begin{align}
		\|A \hat u  - Au^\ast\|^2 &=  \|A \hat u  - P_{A \left(\mathcal B^{\mathcal P'}\right)} f\|^2 \notag \\
		&= \|A \hat u  - f\|^2 - \|f- P_{A \left(\mathcal B^{\mathcal P'}\right)} f\|^2 
		= \|A \hat u  - f\|^2 - \|A  u^\ast  - f\|^2. 		
	\end{align} 
	Inserting $u^\ast$ in the estimate \eqref{eq:estFucValueFrho}, we get 
	\begin{align}\label{eq:estFucValueFrho2}
	F_{\rho}(u') \leq  F_{\rho}(u^\ast,\ldots,u^\ast) \leq  F_{\rho}(\hat u,\ldots,\hat u) 
	\leq \eta + F_{\rho}(u')  \leq \eta + F_{\rho}(u^\ast,\ldots,u^\ast).
	\end{align}
	This allows us to further estimate 
	\begin{align}
	\|A \hat u  - Au^\ast\|^2 &= \|A \hat u  - f\|^2 - \|A  u^\ast  - f\|^2 
	\leq \|A u^\ast  - f\|^2+ \eta  - \|A  u^\ast  - f\|^2 = \eta.
	\end{align}
	If now the operator $A$ is lower bounded, then 
	\begin{align}
	\|\hat u  - u^\ast\|^2 < \frac{1}{c^2}  \|A \hat u  - Au^\ast\|^2 \leq \frac \eta {c^2}
	\end{align}
	which completes the proof.
\end{proof}

\subsection{Majorization-minimization for multivariate Potts problems} \label{sec:MajoMin}

In this part we build the basis for the convergence analysis of Algorithm~\ref{alg:Algo1} and Algorithm~\ref{alg:Algo2}.

We first recall some basics on surrogate functionals.
We consider functionals 
$F(u)$ of the form $F(u) = \|Xu-z\|^2+ \gamma J(u),$ where $X$ is a given (measurement) matrix with 
operator norm $\|X\|<1$ (with the operator norm formed w.r.t.\ the $\ell^2$ norm),
 $z$ is a given vector (of data), 
 $J$ is an arbitrary (not necessarily convex) lower semicontinuous functional, 
 and $\gamma>0$ is a parameter.  
In general, the surrogate functional $F^{\rm surr}(u,v)$ of $F(u)$ is given by
\begin{equation}\label{eq:DefSurr}
	F^{\rm surr}(u,v) = F(u) + \|u-v\|^2 - \|Xu-Xv\|^2. 
\end{equation}

\begin{lemma}\label{lem:Nonincreasing}
	Consider the functionals $F(u) = \|Xu-z\|^2+ \gamma J(u)$ as above with $\|X\|<1.$
	(For our purposes, $J$ is the regularizer 
	$\|D(u)\|_{0,\omega}$ given by \eqref{eq:DefDiffOpZeroNorm}.)  
	Then, we get for the associated surrogate functional $F^{\rm surr}$ given by \eqref{eq:DefSurr} (with $J$ as regularizer), that
	\begin{enumerate}
		\item
		the inequality 
		\[
		F^{\rm surr}(u,v) \geq F(u)
		\]
		holds for all $v;$ and $F^{\rm surr}(u,v) = F(u)$ holds if and only if $u=v;$
		\item 
		the functional values $F(u^k)$ of the sequence $u^k$ given by the surrogate iteration $u^{k+1} = \argmin_u F^{\rm surr}(u,u^k)$ are non-increasing, i.e., 
		\begin{equation}\label{eq:SurrogateDecreasing}
		F(u^{k+1}) \leq F(u^{k});
		\end{equation}
		\item
		the distance between consecutive members of the previous surrogate sequence $u^k$ converges to $0,$ i.e.,
		\begin{equation}\label{eq:DistToZero4Surrogate}
		\lim_{k \to \infty} \|u^{k+1}-u^k\| = 0. 
		\end{equation}
	\end{enumerate}
	\end {lemma}
	We note that --when minimizing $F$-- the condition $\|X\|<1$ can always be achieved by rescaling, i.e.,
	dividing the functional $F$ by a number which is larger than $\|X\|^2.$ 	
	Proofs of the general statements above on surrogate functionals
	(which do not rely on the specific structure of the problems considered here)  
	may for instance be found in the above mentioned papers 
	\cite{daubechies2004iterative,blumensath2008iterative,fornasier2010iterative}.

We now employ properties of the quadratic penalty relaxation 
$P_{\gamma, \rho}(u_1,\ldots,u_S)$ of the Potts energy given by \eqref{eq:MinimizeMultVarRelax}.
The strategy is similar to the authors' approach for the univariate case in \cite{weinmann2015iterative}.
We first show that the minimizers of $P_{\gamma, \rho}(u_1,\ldots,u_S)$ (with $B= \mathrm{id}$ in \eqref{eq:RelaxedProblemInShortNotation}) 
which are precisely the solutions of \eqref{eq:surIt4MultPottsRelaxed}
have a minimal directional jump height which only depends on the scale parameter $\gamma,$
the directional weights $\omega_s$ and the constant $L_{\rho}$
but not on the particular input data. 
Here, for the multivariate discrete function $u= (u_1,\ldots,u_S)$ (and the directional system $a_s,$ $s=1,\ldots,S$) a directional jump is a jump in the $s$th component $u_s$ in direction $a_s$ for some $s.$ In particular, jumps of $u_s$ in directions $a_{s'}$ with $s'\neq s$ are not considered.

\begin{lemma}\label{lem:MinimalJumpHeight}
	We consider the function
	$P_{\gamma, \rho}(u_1,\ldots,u_S)$ of \eqref{eq:RelaxedProblemInShortNotation}
	for the choice $B=\id$ and data $h=(h_1,\ldots,h_S)$.	
	In other words, we consider the problem \eqref{eq:surIt4MultPottsRelaxed} for  
	arbitrary data $h=(h_1,\ldots,h_S)$.
	Then there is a constant $c>0$ which is independent of the minimizer $u^\ast =(u_1^\ast,\ldots,u_S^\ast)$ of \eqref{eq:surIt4MultPottsRelaxed} and the data $h$ 
	such that the minimal directional jump height $j_{\min}(u^\ast)$ (w.r.t. the directional system $a_s,$ $s=1,\ldots,S,$) of a minimzer $u^{\ast}$  fulfills
	\begin{equation}\label{eq:MinimalJumpHight}
	j_{\min}(u^\ast) \geq c.
	\end{equation}
	The constant $c$ depends on $\gamma,$
	the directional weights $\omega_s$ and the constant $L_{\rho}.$
\end{lemma}

\begin{proof}
	Writing $ u = (u_1,\ldots,u_S)$ we restate \eqref{eq:surIt4MultPottsRelaxed} as the problem of minimizing
	\begin{equation}
		P^{\mathrm{id}}_{\gamma/L_{\rho}^2}(u_1,\ldots,u_S) = \left\| u - h \right\|_2^2 + \frac{\gamma}{L_{\rho}^2} \ \Big\| \ D(u_1,\ldots,u_S) \ \Big\|_{0,\omega}
	\end{equation}
	where we use the notation 
	$\|D(u_1,\ldots,u_S)\|_{0,\omega} = \sum _{s=1}^S \omega_s  \left\| \nabla_{a_s} u_s \, \right\|_0$
	introduced in \eqref{eq:DefDiffOpZeroNorm}.
	We let  
	\begin{equation}\label{eq:ChoiseLittleC}
	c = \sqrt{\tfrac{\gamma \ \min_{s \in\{1,\ldots,S\}} \omega_s  }{L_{\rho}^2 W}},
	\end{equation}
	where $W$ denotes the maximal length of the signal $u$ per dimension (e.g., if $u$ denotes an $l \times b$ image, then $W=\max(l,b)$.) 
	We now assume that $h_{\min}(u^\ast) < c,$
	which means that the minimizer $u^\ast$ has a directional jump of height
	smaller than $c.$ 
	For such $u^\ast,$ we construct an element $u'$ with a smaller $P^{\mathrm{id}}_{\gamma/L_{\rho}^2}$ value
	which yields a contradiction since $u^\ast$ is a minimizer of $P^{\mathrm{id}}_{\gamma/L_{\rho}^2}$. 
	To this end, we let $a_s$ be a direction such that the component $u^\ast_s$ of $u^\ast$ has a jump of height smaller than $c.$
	We denote the (discrete) directional intervals in direction $a_s$ near the directional jump by $I_1,I_2$
	and the corresponding points near the directional jump of $u_s^\ast$ by $x_1$ and $x_2.$
	We let $m_1,m_2$ and $m$ be the mean of $h_s$ on $I_1,I_2$ and $I_1 \cup I_2$, respectively.  
	We define 
	\begin{equation} \label{eq:DefUprime}
	u_{s'}' = u_{s'}^\ast  \quad \text{if} \quad s' \neq s,
	\qquad \text{ and } \qquad
	  u_s'(x) = 
	\begin{cases}
	m   & \text{ for } x \in I_1 \cup I_2  \\
	u_s^\ast (x)     & \text{ elsewhere. }
	\end{cases}           
	\end{equation}
	By construction,  $\|\nabla u'\|_0 = \|\nabla u^\ast\|_0 -1.$
	and thus 
	\begin{equation} \label{eq:estJmpsWithweights}
	\|D(u'_1,\ldots,u'_S)\|_{0,\omega}  =    \|D(u^\ast_1,\ldots,u^\ast_S)\|_{0,\omega} - \omega_s
	\leq \|D(u^\ast_1,\ldots,u^\ast_S)\|_{0,\omega} - \min_{s \in\{1,\ldots,S\}} \omega_s.
	\end{equation}
	Since $u^\ast$ is a minimizer of $P^{\mathrm{id}}_{\gamma/L_{\rho}^2},$ 
	its $s$th component $u^\ast_s$ equals $m_1$ on $I_1$ and $m_2$ on $I_2.$ 
	Further, as $u_{s'}' = u_{s'}^\ast$ if $s' \neq s$
	and $u_s^\ast$ and $u_s'$ only differ on $I_1 \cup I_2,$
	we have that
	\begin{align}\label{eq:estNrm}\notag
	\|u'-h\|^2 = \sum_{s'=1}^S \|u_{s'}'-h_{s'}\|^2 &= \sum_{s'=1, s' \neq s }^S \|u_{s'}^\ast-h_{s'}\|^2
	+\|u_s^\ast-h_s\|^2 + l_1 |m_1-m|^2 + l_2 |m_2-m|^2 \\
	&< \|u^\ast-h\|^2 + W  c^2, 
	\end{align}
	where $l_1,l_2$ denote the length of $I_1$ and $I_2,$ respectively.
	Employing \eqref{eq:estJmpsWithweights} together with \eqref{eq:estNrm} we get
	\begin{align*}
	P^{\mathrm{id}}_{\gamma/L_{\rho}^2}(u'_1,\ldots,u'_S) &= \left\|  u' - h \right\|_2^2 + \frac{\gamma}{L_{\rho}^2} \ \Big\| \ D(u'_1,\ldots,u'_S) \ \Big\|_{0,\omega} \\
	& < \|u^\ast-h\|^2 + W  c^2 +  \frac{\gamma}{L_{\rho}^2}\|D(u^\ast_1,\ldots,u^\ast_S)\|_{0,\omega} - \frac{\gamma}{L_{\rho}^2}\min_{s \in\{1,\ldots,S\}} \omega_s \\
	& \leq \|u^\ast-h\|^2 +   \frac{\gamma}{L_{\rho}^2}\|D(u^\ast_1,\ldots,u^\ast_S)\|_{0,\omega} = 
	P^{\mathrm{id}}_{\gamma/L_{\rho}^2}(u^\ast_1,\ldots,u^\ast_S).
	\end{align*} 
	The validity of the last inequality follows by \eqref{eq:ChoiseLittleC}.
	Together, $u'$ has a smaller function value than $u^\ast$
	which is a contradiction to $u^\ast$ being a minimizer which shows the assertion.
\end{proof}

\begin{proposition}\label{pro:RelaxedPottsConverges2LocMinimizer}	
	The iteration \eqref{eq:backwardStepAlg1} of Algorithm~\ref{alg:Algo1} converges to a local minimizer 
	of the quadratic penalty relaxation 
	$P_{\gamma, \rho}$ of the Potts objective function given by \eqref{eq:MinimizeMultVarRelax}.
	The convergence rate is linear.
\end{proposition}

\begin{proof}
	
	We divide the proof into three parts. First, we show that the directional partitionings induced by the iterates $u^{(n)}$ become fixed
	after sufficiently many iterations. In a second part, we derive the convergence of Algorithm~\ref{alg:Algo1}
	and, in a third part, we show that the limit point is a local minimizer of	$P_{\gamma, \rho}$.
	
	(1) We first show that the directional partitioning $\mathcal I^n$ 
	induced by the iterates $u^{(n)}$ gets fixed for large $n$.
	For every $n \in \mathbb N,$ the iterate $u^{(n)}$ of Algorithm~\ref{alg:Algo1} is a minimizer of 
	the function
	$P_{\gamma, \rho}$ of \eqref{eq:RelaxedProblemInShortNotation}
	for the choice $B=\id$ as it appears in  \eqref{eq:surIt4MultPottsRelaxed}.
	Here, the data $h=(h_1,\ldots,h_S)$ is given by \eqref{eq:dataInSurrIT}.
	By Lemma~\ref{lem:MinimalJumpHeight}
	there is a constant $c>0$ which is independent of the particular $u^{(n)} =(u_1^{(n)},\ldots,u_S^{(n)})$ of \eqref{eq:surIt4MultPottsRelaxed} and the data $h$ 
	such that the minimal directional jump height $j_{\min}(u^{(n)})$ fulfills
	\begin{equation}\label{eq:MinimalJumpHight2}
	j_{\min}(u^{(n)}) \geq c  \quad \text{ for all }\quad n \in \mathbb N.
	\end{equation}
	We note that the parameter $\gamma,$
	the directional weights $\omega_s$ and the constant $L_{\rho}$
	which the constant $c$ depends on by Lemma~\ref{lem:MinimalJumpHeight}
	do not change during the iteration of Algorithm~\ref{alg:Algo1}. 
	
	If two iterates $u^{(n)},u^{(n+1)}$ have 
	different induced directional partitionings $\mathcal I^n, \mathcal I^{n+1}$
	their $\ell^\infty$ distance always fulfills $\|u^{(n)}-u^{(n+1)}\|_\infty> c/2$
	since both $u^{(n)},u^{(n+1)}$ have minimal jump height of at least $c$ and different
	induced directional partitionings.
	This implies $\|u^{(n)}-u^{(n+1)}\|_2> c/2$  for the $\ell^2$ distance as well.
	This may only happen for small $n,$ since 	Lemma~\ref{lem:Nonincreasing}
	by \eqref{eq:DistToZero4Surrogate}, we have $\|u^{(n)}-u^{(n+1)}\|_2 \to 0$ as $n$
	increases. Hence, there is an index $N$ such that, for all $n \geq N,$
	the directional partitionings $\mathcal I^n$ are identical.
 
	(2) We use the previous observation to show the convergence of Algorithm~\ref{alg:Algo1}.
	We consider iterates $u^{(n)}$ with $n \geq N;$ 
	they have the same induced directional partitionings which we denote by $\mathcal I',$ 
	and all jumps have minimal jump height $c$.
	Hence, for $n \geq N,$  
	the iteration of \eqref{eq:surIt4MultPottsRelaxed} can be written as 
	\begin{equation}\label{eq:ProjectionAfterFixation}
	u^{(n+1)} = P_{\mathcal I'}(h^{(n)}) 
	\end{equation}
	with $P_{\mathcal I'}$ being the orthogonal projection onto the $\ell^2$ space $\mathcal A^{\mathcal P}$ consisting of functions which are piecewise constant w.r.t.\ the directional partitioning $\mathcal I',$ and where $h^{(n)}$ depends on $u^{(n)}$ via	
	\begin{equation}
	h^{(n)}_s = u^{(n)}_s + \tfrac{1}{SL_{\rho}^2} A^\ast  f - 
	\tfrac{1}{S L_{\rho}^2} A^\ast A u^{(n)}_s - \sum_{s':s' \neq s}\tfrac{\rho_{s,s'}}{L_{\rho}^2}
	(u^{(n)}_s-u^{(n)}_{s'}),  
	\quad \text{ for all } s \in\{1,\ldots,S\},
	\end{equation}
	as given by \eqref{eq:dataInSurrIT}. As introduced before, we use  the symbols $\tilde A$ to denote the block diagonal matrix with constant entry $A$ on each diagonal component, and
	$\tilde f$ for the block vector of corresponding dimensions with entry $f$ in each component.
	With this notation we may write \eqref{eq:ProjectionAfterFixation} as 
	\begin{equation}\label{eq:ProjectionAfterFixation2}
	u^{(n+1)}  = P_{\mathcal I'}((I - \tfrac{1}{SL_{\rho}^2} (\tilde A)^{T} \tilde A- \tfrac{1}{SL_{\rho}^2} \rho \ C^TC) u^{(n)} + \tfrac{1}{SL_{\rho}^2} \tilde A^{T}\tilde f).
	\end{equation}
	Since $u^{(n)}$ is piecewise constant w.r.t.\ the directional partitioning $\mathcal I',$ we have   $u^{(n)} = P_{\mathcal I'} u^{(n)}.$ Using this fact and the fact that $P_{\mathcal I'}$ is an orthogonal projection we obtain 
	\begin{equation}\label{eq:ProjectionAfterFixation3}
	u^{(n+1)}  = \left(I -  
	\left(
	\left(\tfrac{\tilde A P_{\mathcal I'}}{\sqrt S L_{\rho}} \right)^{T}
	\left(\tfrac{\tilde A P_{\mathcal I'}}{\sqrt S L_{\rho}} \right) +
	\left(\tfrac{\sqrt{\rho}C P_{\mathcal I'}}{\sqrt S L_{\rho}} \right)^{T}
	\left(\tfrac{\sqrt{\rho}C P_{\mathcal I'}}{\sqrt S L_{\rho}} \right)
	\right) \right) u^{(n)} 
	 + \left(\tfrac{\tilde A P_{\mathcal I'}}{\sqrt S L_{\rho}} \right)^{T}  \tfrac{\tilde f}{{\sqrt S L_{\rho}}}.
	\end{equation}
	Since $C \tilde A^{T}\tilde f = 0,$
	the iteration \eqref{eq:ProjectionAfterFixation3} can be interpreted as Landweber iteration 
	for the block matrix consisting of the upper block $(\tilde A P_{\mathcal I'})/(\sqrt S L_{\rho})$ and the lower block 
	$(\sqrt{\rho}C P_{\mathcal I'})/(\sqrt S L_{\rho})$
	 and data $\tilde f/(\sqrt S L_{\rho})$ extended by $0.$ 
	 The Landweber iteration converges at a linear rate; cf., e.g., \cite{engl1996regularization}. Thus, the iteration \eqref{eq:ProjectionAfterFixation} convergences
	and, in turn, we get the convergence of Algorithm~\ref{alg:Algo1} at a linear rate to some limit $u^\ast$.   
	
	(3) We show that $u^\ast$ is a local minimizer. 
	Since $u^\ast$ is the limit of the iterates $u^{(n)} $, the jumps of $u^\ast$ also have minimal
	height $c,$ the number of jumps are equal to those of the $u^{(n)}$ for all $n \geq N,$ and the induced directional partitioning $\mathcal I^\ast$ equals the partitioning 
	$\mathcal I'$ of the $u^{(n)}$ for $n \geq N.$
	Since $u^\ast$ equals the limit of the above Landweber iteration,
    $u^\ast$ minimizes $F_\rho$ given by \eqref{eq:QuadraticRelaxOnAp} on $\mathcal A^{\mathcal P_{\mathcal{I}'}}.$
	Then by Lemma~\ref{lem:CharLocMinQuPenRelax}
	$u^\ast$ is a local minimizer 
	of the relaxed Potts energy $P_{\gamma,\rho}$
	which completes the proof.  
\end{proof}

After having shown the convergence of Algorithm~\ref{alg:Algo1}  to a local minimizer, we have now gathered all information to show Theorem~\ref{thm:EffectRelaxedAlgo4OriginalProblem}.
\begin{proof}[Proof of Theorem~\ref{thm:EffectRelaxedAlgo4OriginalProblem}]
	Assertion {\em i.}\ was stated and shown as Proposition~\ref{prop:aproxLocMin} in Section~\ref{sec:quadPenRelaxRelation}.	
	By Proposition~\ref{pro:RelaxedPottsConverges2LocMinimizer} Algorithm~\ref{alg:Algo1} produces a local minimizer. Then the assertion {\em ii.} is a consequence of Proposition~\ref{pro:theoretResOnProc1}.	
\end{proof}

\subsection{Estimating the distance between the objectives}\label{sec:EstimatingDistance}

The next lemma is a preparation for the proof of item \emph{(iii)} of Theorem~\ref{thm:RelaxedConvergesAndTolerance}. 

\begin{lemma}\label{lem:RelaxedConvergesAndTolerance}
	We consider  Algorithm~\ref{alg:Algo1} 
	for the quadratic penalty relaxation \eqref{eq:MinimizeMultVarRelax} of the multivariate Potts problem. For any output $u =(u_1,\ldots,u_S)$ of Algorithm~\ref{alg:Algo1} we have that
		\begin{equation}\label{eq:Tolerance4theUsLemma}
		\left(\sum\nolimits_{s,s'} c_{s,s'} \|u_s - u_{s'}\|^2_2\right)^{\tfrac{1}{2}}  
		\leq  2 \sigma_1^{-1/2} S^{-1/2} \|A\| \|f\|/ \rho.
		\end{equation}
		Here, $\sigma_1$ denotes the smallest non-zero eigenvalue of $C^TC$ with $C$ given by 
		\eqref{eq:CforCoupling}.
\end{lemma}

\begin{proof}
Since $u =(u_1,\ldots,u_S)$ is the output of Algorithm~\ref{alg:Algo1} it is a local minimizer 
of the relaxed Potts problem \eqref{eq:MinimizeMultVarRelax}. In particular, there is a directional partitioning $\mathcal I$ with respect to which $u$ is piecewise constant.
We denote the induced partitioning by $\mathcal P =\mathcal P_{\mathcal I}.$
By Lemma~\ref{lem:estCu} we have 
\begin{align}  
\left(\sum\nolimits_{s,s'} c_{s,s'} \|u_s - u_{s'}\|^2_2\right)^{\tfrac{1}{2}} = \| C  u \| =  \| C P_{\mathcal I} u \| \leq \tfrac{1}{\rho}\|\mu^\ast\|+ \sqrt{\tfrac{F_{\rho}(u)- \min_{x \in \mathcal A^{\mathcal P}} F_{\rho}(x)}{\rho}},
\end{align}
where $\mu^\ast$ is an arbitrary Lagrange multiplier of \eqref{eq:ell2problemOnSegments}.
By Lemma~\ref{lem:EstLagrangeMultipliers} we have that 
$\| \mu^\ast \|   \leq   2 \sigma_1^{-1/2} S^{-1/2} \|A\| \|f\|,$
for any partitioning of the discrete domain $\Omega,$
and in particular for the partitioning $\mathcal P =\mathcal P_{\mathcal I}.$
This shows that
\begin{equation*}
\|Cu\| \leq  2 \sigma_1^{-1/2} S^{-1/2} \|A\| \|f\|/ \rho + \sqrt{\tfrac{F_{\rho}(u)- \min_{x \in \mathcal A^{\mathcal P}} F_{\rho}(x)}{\rho}}.
\end{equation*}
Since $u$ is a local minimizer 
of the relaxed Potts problem \eqref{eq:MinimizeMultVarRelax}, it is 
a minimizer of $F_{\rho}$ on $\mathcal A^{\mathcal P}$
by Lemma~\ref{lem:CharLocMinQuPenRelax}, and the second summand on the right hand side equals zero.
This shows \eqref{eq:Tolerance4theUsLemma} and completes the proof.
\end{proof}

We have now gathered all information necessary to show Theorem~\ref{thm:RelaxedConvergesAndTolerance}.

\begin{proof}[Proof of Theorem~\ref{thm:RelaxedConvergesAndTolerance}]
	Part \emph{(i)} is shown by Proposition~\ref{pro:RelaxedPottsConverges2LocMinimizer}.	
	Concerning \emph{(ii)}
	we first show that any global minimizer of the relaxed Potts energy
	 $P_{\gamma, \rho}$ given by 
	 \eqref{eq:MinimizeMultVarRelax}
	 appears as a stationary point of Algorithm~\ref{alg:Algo1}.
	 To this end we start Algorithm~\ref{alg:Algo1} with a global minimizer $\bar u^\ast=(u_1^\ast,\ldots,u_S^\ast)$ as initialization.
	 Then, we have for all $\bar v=(v_1,\ldots,v_S)$ with $\bar v \neq \bar u^\ast,$
	 \begin{align} \label{eq:Est4Later}
	 P_{\gamma, \rho}^{\rm surr}\left(v_1,\ldots,v_S,u_1^\ast,\ldots,u_S^\ast \right)
	 &= P_{\gamma, \rho}(\bar v) - \|B\bar v-B\bar u^\ast\|^2 + \|\bar v-\bar u^\ast\|^2\\
	 &> P_{\gamma, \rho}(\bar v) \geq P_{\gamma, \rho}(\bar u^\ast) = 
	 P_{\gamma, \rho}^{\rm surr} (\bar u^\ast,\bar u^\ast). \notag
	 \end{align} 
	 Here, $B$ is given by \eqref{eq:DefOfBandg}.
	 The estimate \eqref{eq:Est4Later} means that $\bar u^\ast$ is the minimizer of the surrogate functional w.r.t.\ the first component, i.e., it is the minimizer of the mapping 
	 $\bar v \mapsto$ $P_{\gamma, \rho}^{\rm surr}(\bar v,\bar u^\ast).$
	 Thus, the iterate $\bar u^{(1)}=(u^{(1)}_1,\ldots,u^{(1)}_S)$ of Algorithm~\ref{alg:Algo1} equals
	 $\bar u^\ast$ when the iteration is started with $\bar u^\ast.$ Thus, the global minimizer $\bar u^\ast$ is a stationary point of Algorithm~\ref{alg:Algo1}.
	 It remains to show that each stationary point of Algorithm~\ref{alg:Algo1} is a local minimizer 
	 of the relaxed Potts energy
	 $P_{\gamma, \rho}$. 
	 This has essentially already been done in the proof 
	 of Proposition~\ref{pro:RelaxedPottsConverges2LocMinimizer}: 
	 start the iteration given by \eqref{eq:surIt4MultPottsRelaxed} with a stationary point $u';$	  
	 its limit equals $u'$ and is thus a local minimizer by Proposition~\ref{pro:RelaxedPottsConverges2LocMinimizer}.

	Concerning \emph{(iii)}, we use Lemma~\ref{lem:RelaxedConvergesAndTolerance} to estimate
	\begin{align}
	\left(\sum\nolimits_{s,s'} c_{s,s'} \|u_s - u_{s'}\|^2_2\right)^{\tfrac{1}{2}} \leq 2 \sigma_1^{-1/2} S^{-1/2} \|A\| \|f\|/ \rho  < \varepsilon.       	 
	\end{align}
	The second inequality follows by our choice of $\rho$ in \eqref{eq:chooseRho} as 
	$\rho > 2 \varepsilon^{-1}  \ \sigma_1^{-1/2} S^{-1/2} \|A\| \|f\|.$
	This shows the validity of \emph{(iii)} and completes the proof.
\end{proof}

\subsection{Convergence Analysis of Algorithm~\ref{alg:Algo2}}\label{sec:ConvAlgo2}

We start out showing that Algorithm~\ref{alg:Algo2} is well-defined in the sense that the inner iteration governed by \eqref{eq:TerminationInnerLoopAlg2} terminates. This result was formulated as  
Theorem~\ref{thm:Algo2isWellDefined}.
\begin{proof}[Proof of Theorem~\ref{thm:Algo2isWellDefined}]
	We have to show that, for any $k \in \mathbb N,$ there is $n \in \mathbb N$ such that 
		\begin{align}\label{eq:TerminationInnerLoopAlg2NegProof}
		\left\|u^{(k,n)}_s - u^{(k,n)}_{s'} \right\| \leq \frac{t}{\rho_k \sqrt{c_{s,s'}}}, \quad \text{ and } \quad
		\left\| u^{(k,n)}_s - u^{(k,n-1)}_s \right\| \leq \frac{\delta_k}{L_{\rho}}.						    
		\end{align}		
	To see the right-hand side of \eqref{eq:TerminationInnerLoopAlg2NegProof},
	we notice that, by Proposition~\ref{pro:RelaxedPottsConverges2LocMinimizer},	
	the iteration \eqref{eq:backwardStepAlg1} converges to a local minimizer 
	of the quadratic penalty relaxation 
	$P_{\gamma, \rho}(u_1,\ldots,u_S)$ of the Potts energy.
	The inner loop of Algorithm~\ref{alg:Algo2} precisely computes the iteration \eqref{eq:backwardStepAlg1} (for the penalty parameter $\rho$ which increases with $k.$) 			
	Thus, the distance between consecutive iterates $u^{(k,n)}_s,u^{(k,n-1)}_s$ converges to zero as $n$ increases which implies the validity of the right-hand side of \eqref{eq:TerminationInnerLoopAlg2NegProof} for sufficiently large $n,$ and all $k \in \mathbb N.$
	
	To see the left-hand inequality in \eqref{eq:TerminationInnerLoopAlg2NegProof}, 
	we notice that, by the considerations above, the inner loop of Algorithm~\ref{alg:Algo2}
    would converge to a minimizer $\bar u^{(k),\ast} = (u_1^{(k),\ast},\ldots,u_S^{(k),\ast})$ 
	if it was not terminated by \eqref{eq:TerminationInnerLoopAlg2NegProof} for all $k \in \mathbb N.$
	Since $\bar u^{(k),\ast}$ is a local minimizer 
	of the relaxed Potts problem \eqref{eq:MinimizeMultVarRelax} for the parameter $\rho_k$, it is 
	a minimizer of $F_{\rho_k}$ on $\mathcal A^{\mathcal P}$ (where $\mathcal P$ denotes the partitioning induced by $\bar u^{(k),\ast}$)
	by Lemma~\ref{lem:CharLocMinQuPenRelax}.	
	Hence, for any $k\in \mathbb N$ and any $\xi>0$ there is 
	$\bar u^{(k,n)} =(u^{(k,n)}_1,\ldots,u^{(k,n)}_S)$
	such that 
	$F_{\rho_k}(\bar u^{(k,n)})- F_{\rho_k}(\bar u^{(k),\ast})  <  \xi.$
	We let 
	$\tau = (t - 2\sigma_1^{-1/2} S^{-1/2} \|A\| \ \|f\|) /\rho_k,$
	and choose $\xi = \rho_k \tau^2.$
	Using this together with  Lemma~\ref{lem:estCu} we estimate 
	\begin{align} \label{eq:estbyTbyRho} 
	\sqrt{c_{s,s'}} \|u_s^{(k,n)} - u_{s'}^{(k,n)}\|_2 = \| C  \bar u^{(k,n)} \| 
	&\leq \tfrac{1}{\rho_k}\|\mu^\ast\|+ \sqrt{\tfrac{F_{\rho_k}(\bar u^{(k,n)})- F_{\rho_k}(\bar u^{(k),\ast})}{\rho_k}} \notag\\
	&\leq \tfrac{1}{\rho_k}\|\mu^\ast\| + \sqrt{\tfrac{\xi}{\rho_k}}
	\leq \tfrac{1}{\rho_k}\|\mu^\ast\| + \tau \leq \tfrac{t}{\rho_k} 
	\end{align}
	where $\mu^\ast$ is an arbitrary Lagrange multiplier of \eqref{eq:ell2problemOnSegments}.
	Here, the last inequality is true since by Lemma~\ref{lem:EstLagrangeMultipliers} we have that 
	$\| \mu ^\ast\|  \leq   2 \sigma_1^{-1/2} S^{-1/2} \|A\| \|f\|$
	which implies that $\tau \leq (t-\mu^\ast)/\rho_k.$
	The estimate \eqref{eq:estbyTbyRho} shows the left-hand inequality in \eqref{eq:TerminationInnerLoopAlg2NegProof}
	and completes the proof.	
\end{proof}

We have now gathered all information to prove Theorem~\ref{thm:ConvAlgo2} which
deals with the convergence properties of Algorithm~\ref{alg:Algo2}. 
\begin{proof}[Proof of Theorem~\ref{thm:ConvAlgo2}]
		We start out to show that any accumulation point of the sequence $u^{(k)}$ 
		produced by Algorithm~\ref{alg:Algo2}
		is a local minimizer of the Potts problem \eqref{eq:MinimizeMultVar}. 
		Let $u^\ast$ be such an accumulation point and let $\mathcal I^\ast$ be the directional 
		partitioning induced by $u^\ast.$
		We may extract a subsequence $u^{(k_l)}$ of the sequence $u^{(k)}$ such that
		$u^{(k_l)}$ converges to $u^\ast$ as $l \to \infty,$
		and such that the directional partitionings $\mathcal I^{k_l}$ induced by the 
		$u^{(k_l)}$ all equal the directional partitioning $\mathcal I^\ast,$ 
		i.e., $\mathcal I^{k_l}=\mathcal I^\ast$ for all $l \in \mathbb N.$
		We let
		\begin{equation}\label{eq:Defmukl}
			\mu^{k_l} =  - 2 \rho_{k_l} \ C u^{k_l}
		\end{equation} 
        with the matrix $C$ given by \eqref{eq:CforCoupling},		
		and estimate
		\begin{align}\label{eq:ToVerLagCond}\notag
			\|\tfrac{2}{S}\tilde A^T \tilde A u^{k_l}  - \tfrac{2}{S}\tilde A^T \tilde f-  C^T \mu^{k_l} \|
			&= \|\tfrac{2}{S}\tilde A^T \tilde A u^{k_l} - \tfrac{2}{S}\tilde A^T \tilde f+ 2 \rho_{k_l} C^T  \ C u^{k_l}\| \\
			&= \| \nabla F_{\rho_{k_l}}(u^{k_l}) \|
			\leq \tfrac{\delta_{k_l}}{L_{\rho_{k_l}}} \leq \delta_{k_l}. 
		\end{align}
		We recall that $\tilde A$ was the block diagonal matrix having the matrix $A$ as entry in each diagonal component and that $F_{\rho_{k_l}}$ was given by \eqref{eq:QuadraticRelaxOnAp}. 
		We notice that the second before last inequality follows by the right hand side of \eqref{eq:TerminationInnerLoopAlg2NegProof}. We further estimate
		\begin{equation*}
			\|\mu^{k_l}\| = \rho_{k_l} \|C u^{k_l}\| < \rho_{k_l} \tfrac{S t}{\rho_{k_l}} = S t
		\end{equation*}
		which is a consequence of the left hand side of \eqref{eq:TerminationInnerLoopAlg2NegProof}.
		Hence, the sequence $\mu^{k_l}$ is bounded and thus has a cluster point, say $\mu^\ast,$ by the 
		Bolzano-Weierstra\ss \ Theorem. By passing to a further subsequence (where we suppress the new indexation for better readability and still use the symbol $l$ for the index)
		we get that  
		\begin{equation}\label{eq:LamusConverge}
		\mu^{k_l} \to \mu^\ast \quad \text{ as }\quad  l \to \infty.	
		\end{equation}
		Now, on this subsequence, we have that $u^{(k_l)} \to u^{\ast}$ and that 
		$\mu^{k_l} \to \mu^\ast.$ 
		Hence taking limits on both sides of \eqref{eq:ToVerLagCond} yields
		\begin{align}\label{eq:ToVerLagCond2}
		\tfrac{2}{S}\tilde A^T \tilde A u^{\ast}  - \tfrac{2}{S}\tilde A^T \tilde f
		-  C^T \mu^{\ast} = 0, 
		\end{align}
		since $\delta_{k_l} \to 0$ as $l \to \infty.$
		Further,
		\begin{equation}
			\|Cu^\ast\| \leq \lim_{l \to \infty}\tfrac{\|\mu^{k_l}\|}{\rho_{k_l}} 
			\leq \| \mu^\ast \|   \lim_{l \to \infty}\tfrac{1}{\rho_{k_l}} =0.
		\end{equation}
		This implies that the components of $u^\ast$ are equal, i.e., $u_s^{\ast} = u_{s'}^{\ast}$ for all $s,s'.$ In particular $u^\ast$ is a feasible point for the Potts problem \eqref{eq:MinimizeMultVar}. Or, letting $\mathcal P^\ast$ be the partitioning induced by $u^\ast,$ we have that $u^\ast \in \mathcal B^{\mathcal P}.$
 		Then, \eqref{eq:ToVerLagCond2} tells us that $u^\ast$ minimizes \eqref{eq:ell2problemOnSegments}
 		which by Lemma~\ref{lem:CharLokMinPotts} tells us that $u^\ast$ is a local minimizer of 
 		\eqref{eq:MinimizeMultVar}, or synonymously, 
 		that any component of $u^\ast$ (which are all equal)
 		minimizes  the Potts problem \eqref{eq:PottsDiscrete}.
 		This shows the first assertion of Theorem~\ref{thm:ConvAlgo2}.

		We continue by showing the second assertion of Theorem~\ref{thm:ConvAlgo2}, i.e.,
		if  $A$ is lower bounded, then
		the sequence $u^{(k)}$ produced by Algorithm~\ref{alg:Algo2} 
		has a cluster point. Then, by the above considerations, each cluster point is a local minimizer which shows the assertion. To this end, we show that, 
		if $A$ is lower bounded, the sequence $u^{(k)}$ produced by Algorithm~\ref{alg:Algo2} is bounded which by the Heine-Borel property of finite dimensional Euclidean space implies that 
		it has a cluster point. 
		So we assume that $A$ is lower bounded,
		and consider the sequence $u^{(k)}=(u_1^{(k)},\ldots,u_S^{(k)})$ produced by Algorithm~\ref{alg:Algo2}.
		As in the proof of Theorem~\ref{thm:Algo2isWellDefined} we see that, for any $k \in \mathbb N$,  
		there is a local minimizer $ u^{(k),\ast} = (u_1^{(k),\ast},\ldots,u_S^{(k),\ast})$ of 
		\eqref{eq:MinimizeMultVarRelax} such that
		\begin{equation}\label{eq:IterateIsNearMinimizer}
			\|u^{(k)}- u^{(k),\ast}\| \leq C_2 \delta_k,
		\end{equation}	
		where $C_2$ is a constant independent of $k.$	
		By Lemma~\ref{lem:CharLocMinQuPenRelax}, 
		$ u^{(k),\ast}$ is a minimizer of $F_{\rho}$ on $\mathcal A^{\mathcal P}$ (where $\mathcal P$ denotes the partitioning induced by $ u^{(k),\ast}$.)
		Hence,
		\begin{equation*}
			\tfrac{1}{S} \sum\nolimits_{s=1}^S\| A u_s^{(k),\ast} -  f \|^2 \leq  F_{\rho}( u^{(k),\ast}) \leq 
			\|f \|^2
		\end{equation*}
		by choosing the function having the zero function as entry in each component as a candidate.
		This implies 
		\begin{equation}\label{eq:BoundMinimizer}
		\tfrac{1}{S} \sum\nolimits_{s=1}^S\| A u_s^{(k),\ast} \|^2 \leq 
		4 \|f \|^2.
		\end{equation} 
		Then, since $A$ is lower bounded, there is a constant $c>0$ such that
		\begin{equation}
			\|u^{(k),\ast}\|^2 
			= \tfrac{1}{S} \sum\nolimits_{s=1}^S\| u_s^{(k),\ast}\|^2
			\leq  \tfrac{1}{S} \sum\nolimits_{s=1}^S c^2\| A u_s^{(k),\ast}\|^2
			\leq 4 c^2 \|f \|^2
		\end{equation}
		where we used \eqref{eq:BoundMinimizer} for the last inequality. 
		Combining this estimate with \eqref {eq:IterateIsNearMinimizer} yields 
		\begin{equation}\label{eq:IterateIsBounded}
		\|u^{(k)}\| \leq \|u^{(k)}- u^{(k),\ast}\| + \|u^{(k),\ast}\| \leq C_2 \delta_k + 2 c \|f\|.		
		\end{equation}			
		Since we have chosen $\delta_k$ as a sequence converging to zero, 
		\eqref{eq:IterateIsBounded} shows that the sequence $u^{(k)}$ is bounded
		which implies that it has cluster points. This completes the proof. 
\end{proof}

\section{Numerical Results}
\label{sec:NumericalExperiments}
In this section, we show the applicability of our methods to 
different imaging tasks.
We start out by providing the necessary implementation details. 
Then we compare the results of 
the quadratic relaxation \eqref{eq:MinimizeMultVarRelax}
(Algorithm \ref{alg:Algo1})
to the ones of the Potts problem 
\eqref{eq:PottsDiscrete}
(Algorithm \ref{alg:Algo2}).
Next, we apply Algorithm \ref{alg:Algo2} to blurred image data and 
to image reconstruction from incomplete Radon data.
Finally, we consider the 
image partitioning problem
according to the classical Potts model.
\paragraph{Implementation details.}
We implemented Algorithm \ref{alg:Algo1} and Algorithm \ref{alg:Algo2} for the coupling
schemes in \eqref{eq:ChoiceRhos} and the set of compass and diagonal directions
$ (1,0),(0,1),(1,1),(1,-1)$ with weights
$\omega_{1,2} = \sqrt{2}-1$ and $\omega_{3,4} = 1-\frac{\sqrt{2}}{2}$.

Concerning Algorithm \ref{alg:Algo1}
we observed both visually and quantitatively appealing
results if
we use relaxed step-sizes
$L_\rho^\lambda= L_\rho [\lambda + (1-(n+1)^{-1/2})(1-\lambda) ]$ for
an empirically chosen 
parameter $0<\lambda\leq 1$, where
$L_\rho$ denotes the estimate in Lemma \ref{lem:BbyLcontractive}. 
The iterations were stopped when the nearness condition \eqref{eq:maxNormTol}
was fulfilled and the iterates did not change anymore, i.e., when
$\|u_1^{(n)} - u_1^{(n-1)}\|/(\|u_1^{(n)}\| + \|u_1^{(n-1)}\|)$ 
and 
$\|u_2^{(n)} - u_2^{(n-1)}\|/(\|u_2^{(n)}\| + \|u_2^{(n-1)}\|)$
were smaller than  $10^{-6}$.
The result of Algorithm \ref{alg:Algo1} was transformed into a feasible solution
of 
\eqref{eq:MinimizeMultVar}
by applying the projection procedure
described in Section~\ref{sec:ProposedSchemes}
(Procedure \ref{proc:ProjProc}).
As initialization we applied 1000 Landweber iterations with step-size $1/\|A\|^2$
to the least squares problem induced by the linear operator $A$ and data $f$.

Concerning Algorithm \ref{alg:Algo2}, we set $\rho^{(0)} = 10^{-3}$
in all experiments
which we incremented 
by the factor $\tau = 1.05$ in each outer iteration.
The $\delta$-sequence was chosen as
$\delta^{(k)} = \frac{1}{\eta\rho^{(k)}}$ 
for $\eta = 0.95$ when coupling all variables
and $\eta = 0.98$ when coupling consecutive variables. 
Similarly to Algorithm \ref{alg:Algo1},
step A of Algorithm \ref{alg:Algo2} was performed using
the relaxed step sizes  
$L_\rho^\lambda= L_\rho [\lambda + (1-(n+1)^{-1/2})(1-\lambda) ]$
for an application-dependent
parameter $0<\lambda\leq 1$ and for
the estimate $L_\rho$  in Lemma \ref{lem:BbyLcontractive}. 
We stopped the iterations when the relative discrepancy of the first two
splitting variables 
$\| u^{(k)}_1 - u^{(k)}_2 \| / (\| u^{(k)}_1 \| + \| u^{(k)}_2 \|) $
was smaller than  $10^{-6}$. 
We initialized Algorithm \ref{alg:Algo2} with $A^Tf$.

\paragraph{Comparison of Algorithm 1 and Algorithm 2.}
We compare Algorithm \ref{alg:Algo1} and Algorithm \ref{alg:Algo2} for blurred image data,
that is, the linear operator $A$ 
in \eqref{eq:pottsGeneralA}
amounts to the convolution by a kernel $K$. In the present experiment, we chose a Gaussian kernel with standard deviation $\sigma=3
$ 
and of size $6\sigma + 1$.
Here, we coupled all splitting variables and chose the step-size parameter
$\lambda=0.4$ for Algorithm \ref{alg:Algo1}
and $\lambda = 0.35$ for Algorithm \ref{alg:Algo2}, respectively.  
In Figure~\ref{fig:Alg1vsAlg2} we applied both methods to a blurred natural image.
While both algorithms yield reasonable partitionings, 
Algorithm~\ref{alg:Algo2} provides smoother edges than Algorithm~\ref{alg:Algo1}.
Further, Algorithm~\ref{alg:Algo1} 
	produces some smaller segments (at the treetops).

\begin{figure}[!t]
	\def\CompSize{0.24}
	\captionsetup{justification=centering}		
	\begin{subfigure}{\CompSize\textwidth}
		\includegraphics[width=\textwidth]{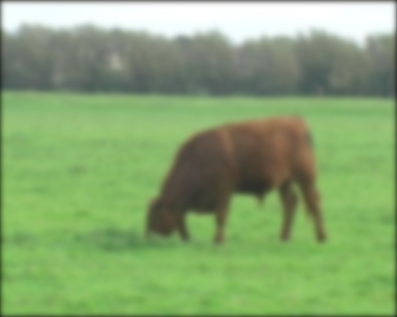}
		\caption{Blurred data\\~}
	\end{subfigure}\hfill
	\begin{subfigure}{\CompSize\textwidth}
		\includegraphics[width=\textwidth]{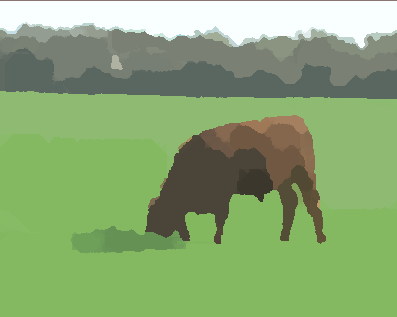}
		\caption{Result of Algorithm~1, \\$\gamma=\protect\input{Alg1vsAlg2/gamma.tex}$,
			$\varepsilon=\protect\input{Alg1vsAlg2/eps.tex}$
		}
	\end{subfigure}\hfill
	\begin{subfigure}{\CompSize\textwidth}
		\includegraphics[width=\textwidth]{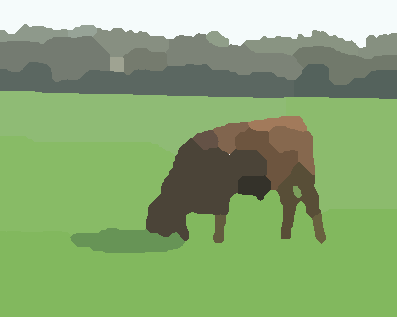}
		\caption{Result of Algorithm~2,\\
			$\gamma=\protect\input{Alg1vsAlg2/gamma.tex}$
		}
	\end{subfigure}\hfill	
	\begin{subfigure}{\CompSize\textwidth}
		\centering
		\captionsetup{justification=centering}		
		\includegraphics[width=\textwidth]{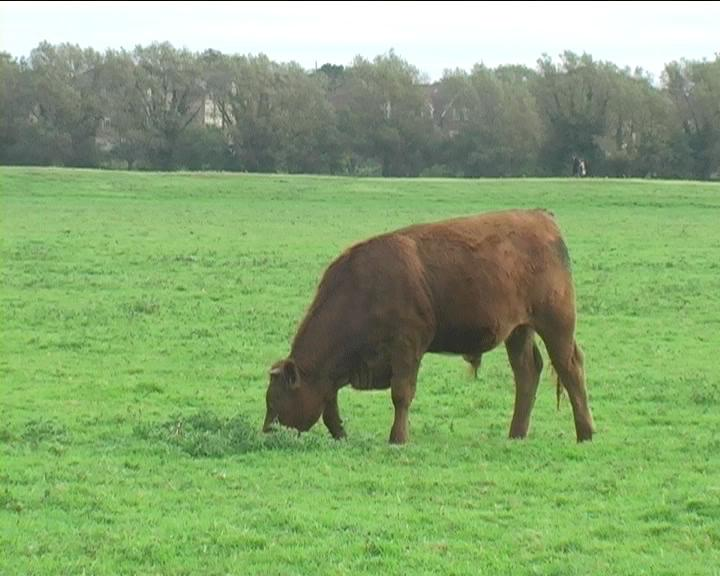}
		\caption{Original \cite{juan2006active}\\~}
	\end{subfigure}\hfill
	\captionsetup{justification=justified}		
	\caption{\label{fig:Alg1vsAlg2}Results of Algorithm~\ref{alg:Algo1} 
		and Algorithm~\ref{alg:Algo2} for partitioning an image blurred
		by a Gaussian kernel of standard deviation
		\protect\input{Alg1vsAlg2/sigma_blurr.tex}\unskip.
		Both methods provide reasonable partitionings. 
		Algorithm~\ref{alg:Algo2} provides smoother edges than
		Algorithm~\ref{alg:Algo1}
		(e.g., the boundary between the meadow and the forest, the back of the cow).
		In addition, Algorithm~\ref{alg:Algo1} produces some smaller segments around the treetops.			
	}
\end{figure}

\paragraph{Application to blurred data.}
For the following experiments, we focus on Algorithm \ref{alg:Algo2}.
In case of motion blur we set the step-size parameter to $\lambda = 0.25$, while for
Gaussian blur we set $\lambda=0.35$ as in Figure \ref{fig:Alg1vsAlg2}.
We compare our method with the
Ambrosio-Tortorelli 
approximation \cite{ambrosio1990approximation} of
the 
classical Mumford-Shah model (which itself tends to the piecewise constant Mumford-Shah model for increasing variation penalty) given by
\begin{equation}\label{eq:AmbrosioTortorelli}
\begin{split}
A_\varepsilon (u,v) = \gamma \int \varepsilon \vert \nabla v \vert^2
+\frac{(v-1)^2}{4\varepsilon} \mathrm{d}x
+\alpha \int v^2 \| \nabla u \|^2 \mathrm{d}x + 
\frac{1}{2} \int (K \ast u - f) \mathrm{d}x.
\end{split}
\end{equation}
The variable $v$ serves as an edge indicator and $\varepsilon>0$ 
is an edge smoothing parameter that is chosen empirically.
The parameter $\gamma > 0$ controls the weight of the edge length
penalty and 
the parameter $\alpha > 0$ penalizes the variation. 
In this respect, a higher value of $\alpha$	
promotes solutions which are closer to being piecewise constant.
In the limit  $\alpha\to \infty$ minimizers of
\eqref{eq:AmbrosioTortorelli} are piecewise constant.
Our implementation follows the scheme presented in \cite{bar2006semi}.
The functional $A_\varepsilon$ is alternately minimized w.r.t. $u$ and $v$.
To this end, we  iteratively solve the Euler-Lagrange equations 
\begin{equation}\label{eq:EulerLagrangeAT}
\begin{split}
2\alpha v \| \nabla u \|_2^2 + \gamma \frac{v-1}{2\varepsilon} - 2\varepsilon\gamma \nabla^2 v &= 0, \\
(K\ast u - f) \ast \widetilde{K} - 2\alpha \mathrm{div}(v^2 \nabla u) &= 0,
\end{split}
\end{equation}
where $\widetilde{K}(x) = K(-x)$. 
The first line is solved w.r.t. $v$ using a MINRES solver and the second line
is solved using the method of conjugate gradients \cite{bar2006semi}.
The iterations were stopped when 
the relative change of both variables was small, i.e.,
if both
$\| u^{k+1} - u^k \| /(\| u^k\| + 10^{-6}) <10^{-3}$
and
$\| v^{k+1} - v^k \| /(\| v^k\| + 10^{-6}) <10^{-3}$.

Figure \ref{fig:motionBlurr} shows the restoration of
a traffic sign
from simulated horizontal motion blur.
For the Ambrosio-Tortorelli approximation we set
$\alpha = 10^5$
to promote a piecewise constant solution.
We observe that both 
the Ambrosio-Tortorelli approximation
and the proposed method restore the data
to a human readable form.
However, the Ambrosio-Tortorelli result
shows clutter and blur artifacts.
Our method provides sharp edges and it produces less artifacts.

In Figure~\ref{fig:gaussianBlurr} we partition a 
natural image blurred by a Gaussian kernel
and corrupted by Gaussian noise.
We observed that the Ambrosio-Tortorelli result was
heavily corrupted by artifacts for $\alpha=10^5$. This might be attributed
to the underlying linear systems in scheme \eqref{eq:EulerLagrangeAT} which
become severely ill-conditioned
for large choices of the variation penalty $\alpha$.
Therefore, we chose the moderate variation penalty
$\alpha=5
$ which 
does only provide an approximately piecewise constant (rather piecewise smooth) result. 
The result does not fully separate the background from the fish in terms of edges.
On the other hand, the result of the proposed method
sharply differentiates between background and fish.
Further, it highlights various segments of the fish.

\begin{figure}[t]
	\def\FigDist{1.3em}
	\def\BlurrSize{0.24}
	\begin{subfigure}{\textwidth}
		\begin{subfigure}{\BlurrSize\textwidth}
			\includegraphics[width=\textwidth]{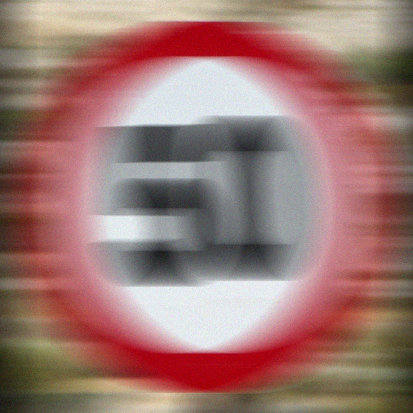}
			\caption{Blurred and noisy data\\~}
		\end{subfigure}\hfill
		\begin{subfigure}{\BlurrSize\textwidth}
			\includegraphics[width=\textwidth]{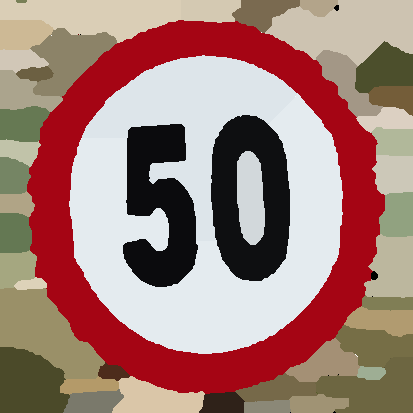}
			\caption{Proposed, $\gamma=\protect\input{Blurr_Motion/gamma.tex}$\\~
			}
		\end{subfigure}\hfill		
		\begin{subfigure}{\BlurrSize\textwidth}
			\includegraphics[width=\textwidth]{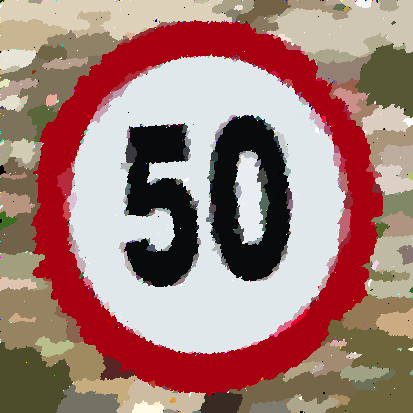}
			\caption{Ambrosio-Tortorelli,\\
				$\alpha=\protect\input{Blurr_Motion/alpha_AT.tex}$,				
				$\gamma=\protect\input{Blurr_Motion/gamma_AT.tex}$,
				$\varepsilon=\protect\input{Blurr_Motion/eps_AT.tex}$
			}
		\end{subfigure}\hfill	
		\begin{subfigure}{\BlurrSize\textwidth}
			\centering
			\captionsetup{justification=centering}		
			\includegraphics[width=\textwidth]{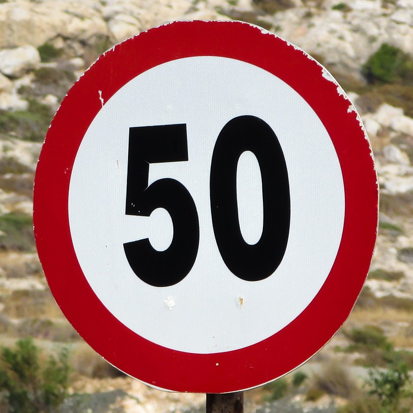}
			\caption{Original\\~}
		\end{subfigure}\hfill		
	\end{subfigure}\hfill			
	\caption{\label{fig:motionBlurr}
		Restoration
		from simulated horizontal
		motion blur of 
		\protect\input{Blurr_Motion/mask_size.tex}pixel length
		and Gaussian noise with 
		$\sigma=\protect\input{Blurr_Motion/sigma_noise.tex}$\unskip.
		The result of the Ambrosio-Tortorelli scheme exhibits
		noisy and blurred artifacts 
		and
		bumpy edges (e.g., the boundaries of the digits).
		The contours of the proposed result are concise
		and considerably less clutter is present.		
	}	\vspace{\FigDist}
	\begin{subfigure}{\textwidth}
		\captionsetup{justification=centering}		
		\begin{subfigure}{\BlurrSize\textwidth}
			\includegraphics[width=\textwidth]{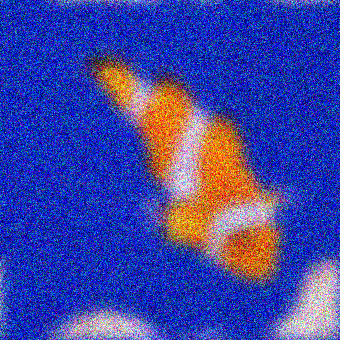}
			\caption{Blurred and noisy data\\~}
		\end{subfigure}\hfill	
		\begin{subfigure}{\BlurrSize\textwidth}
			\includegraphics[width=\textwidth]{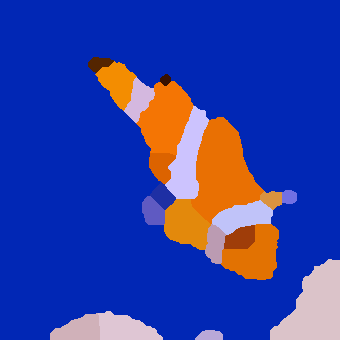}
			\caption{Proposed, $\gamma=\protect\input{Blurr_Gaussian/gamma.tex}$\\~
			}
		\end{subfigure}\hfill
		\begin{subfigure}{\BlurrSize\textwidth}
			\includegraphics[width=\textwidth]{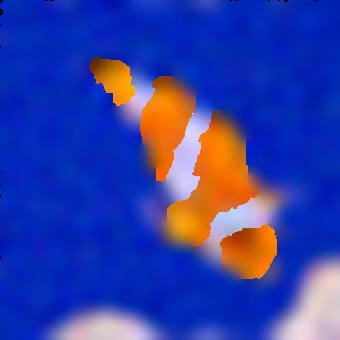}
			\caption{Ambrosio-Tortorelli, $\alpha=\protect\input{Blurr_Gaussian/alpha_AT.tex}$,
				$\gamma=\protect\input{Blurr_Gaussian/gamma_AT.tex}$,
				$\varepsilon=\protect\input{Blurr_Gaussian/eps_AT.tex}$
			}
		\end{subfigure}\hfill		
		\begin{subfigure}{\BlurrSize\textwidth}
			\includegraphics[width=\textwidth]{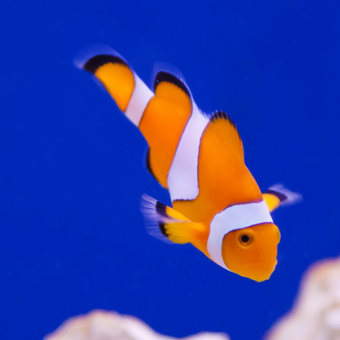}
			\caption{Original\\~}
		\end{subfigure}\hfill	
	\end{subfigure}	
	\caption{\label{fig:gaussianBlurr}
		Partitioning of an image blurred by a Gaussian kernel of standard deviation 
		\protect\input{Blurr_Gaussian/sigma_blurr.tex}\unskip
		~and corrupted by Gaussian noise
		with $\sigma=\protect\input{Blurr_Gaussian/sigma_noise.tex}$\unskip.
		The result of the Ambrosio-Tortorelli approximation does not
		yield a convincing partitioning of the scene, in particular many parts of the fish are merged with the background.
		The proposed approach provides a partitioning which reflects many parts of the fish.
	} 		
\end{figure}

\paragraph{Reconstruction from Radon data.}
\begin{figure}[ht]
	\def\TomoSize{0.24}
	\begin{subfigure}{\textwidth}	
		\begin{subfigure}{\textwidth}
			\centering
			\captionsetup{justification=centering}
			\begin{subfigure}{\TomoSize\textwidth}
				\includegraphics[width=\textwidth]{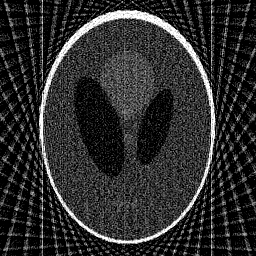}
				\caption{FBP (Ram-Lak Filter)\\MSSIM: \protect\input{Tomo/mssim_FBP.tex}
				}		
			\end{subfigure}\hfill
			\begin{subfigure}{\TomoSize\textwidth}
				\includegraphics[width=\textwidth]{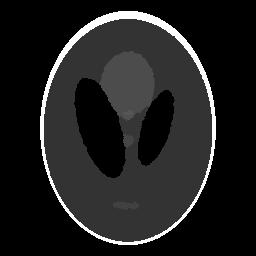}
				\caption{Proposed, $\gamma=\protect\input{Tomo/gamma.tex}$\\~
					MSSIM: \protect\input{Tomo/mssim_ours.tex}
				}
			\end{subfigure}\hfill
			\begin{subfigure}{\TomoSize\textwidth}
				\includegraphics[width=\textwidth]{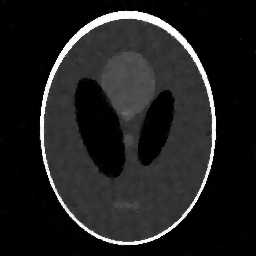}
				\caption{TV, $\mu=\protect\input{Tomo/gamma_TV.tex}$\\~
					MSSIM: \protect\input{Tomo/mssim_TV.tex}
				}
			\end{subfigure}\hfill
			\begin{subfigure}{\TomoSize\textwidth}
				\includegraphics[width=\textwidth]{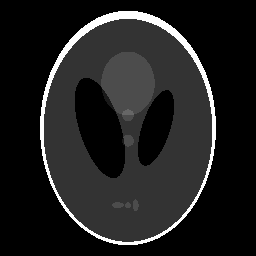}
				\caption{Original\\~}
			\end{subfigure}\hfill					
		\end{subfigure}	
	\end{subfigure}	
	\caption{\label{fig:radon}
		Reconstruction of the Shepp–Logan phantom from 
		undersampled Radon data
		(\protect\input{Tomo/nAngles.tex}\unskip~projection angles)
		corrupted by Gaussian noise with $\sigma=0.7$.	
		The proposed method provides a genuine piecewise constant reconstruction
		and the SSIM is improved by the factor 11.58
		for filtered backprojection
		and by 1.05 for total variation, respectively.			
	}
\end{figure}
We here consider reconstruction from Radon data which appears for instance
in computed tomography. We recall that the Radon transform reads
\begin{equation}
\begin{split}
Ru(\theta,s) = \int_{-\infty}^\infty u(s\theta + t\theta^\perp) \mathrm{d}t,
\end{split}
\end{equation}
where $s\in\R,$ $\theta\in S^1$ and $\theta^\perp\in S^1$ is (counterclockwise) perpendicular 
to $\theta$; see \cite{natterer1986mathematics}.
For our experiments, we use a discretization of the Radon transform created using the AIR tools software package \cite{hansen2012air}.
Regarding our method, we employed coupling of consecutive splitting variables and
the step-size parameter was set to $\lambda=0.11$.
To quantify the reconstruction quality we use the mean structural 
similarity index (MSSIM) \cite{wang2004image} which is bounded from above by 1, where
higher values indicate better results.

We compare the proposed method to 
filtered back projection (FBP) which is standard in practice
\cite{pan2009commercial}.
The FBP is computed using its Matlab implementation
with the standard Ram-Lak filter.
Furthermore, we compare with total variation (TV) regularization \cite{rudin1992nonlinear}
in the Lagrange form $\|Ru-f\|_2^2 + \mu \| \nabla u \|_1$ with parameter $\mu>0$.
Its implementation follows the Chambolle-Pock algorithm \cite{chambolle2011first}.
The corresponding parameter $\mu$ was tuned w.r.t.\ the MSSIM index.

In Figure~\ref{fig:radon} we show the reconstruction results 
for the Shepp-Logan phantom
from undersampled (25 angles) and noisy Radon data.
Standard FBP produces strong streak artifacts which are typical for angular undersampling,
and the reconstruction suffers from noise. 
The TV regularization and the proposed method both provide considerably improved reconstruction results. 
The proposed method achieves a higher MSSIM value than the TV reconstruction,
and it provides a reconstruction which is less grainy than the TV result.

\paragraph{Image partitioning.}
\begin{figure}[ht]
	\centering
	\captionsetup{justification=centering}
	\begin{subfigure}{0.32\textwidth}
		\includegraphics[width=\textwidth]{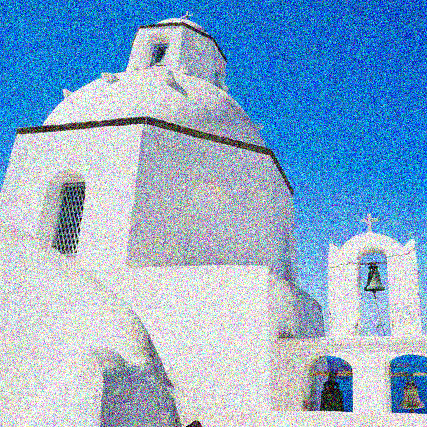}
		\caption{Noisy data}
	\end{subfigure}\hfill	
	\begin{subfigure}{0.32\textwidth}
		\includegraphics[width=\textwidth]{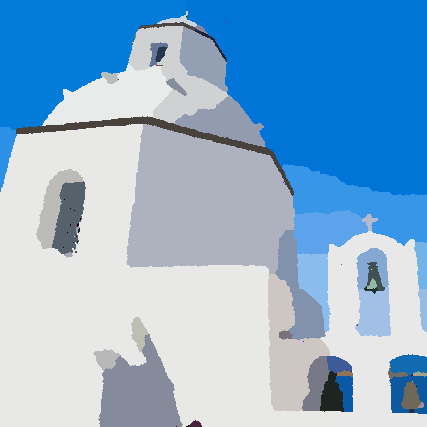}
		\caption{Proposed,
			$\gamma=\protect\input{Partitioning/gamma.tex}$
		}
	\end{subfigure}\hfill
	\begin{subfigure}{0.32\textwidth}
		\includegraphics[width=\textwidth]{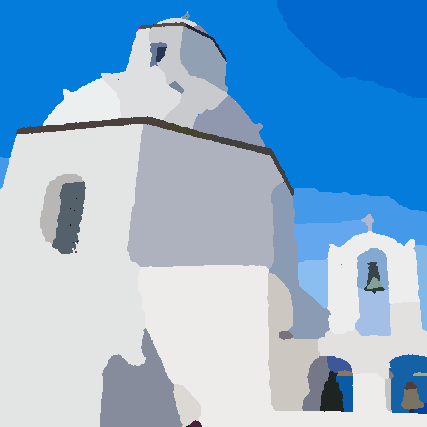}
		\caption{Graph cuts, \protect\input{Partitioning/GC_numLabels.tex}
			\unskip\,labels,
			$\gamma=\protect\input{Partitioning/gamma.tex}$
		}
	\end{subfigure}\\[0.5em]
	\begin{subfigure}{0.32\textwidth}
		\includegraphics[width=\textwidth]{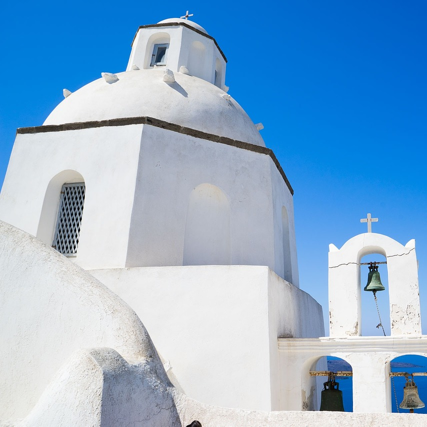}
		\caption{Original\\~}
	\end{subfigure}\hfill	
	\begin{subfigure}{0.32\textwidth}
		\includegraphics[width=\textwidth]{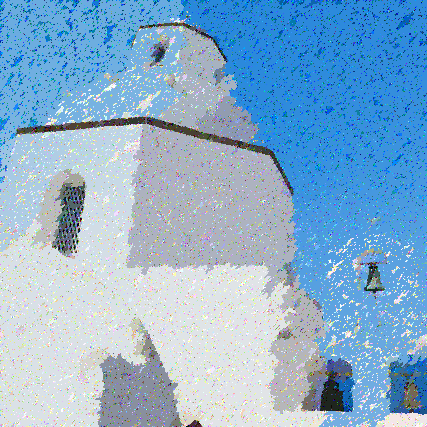}
		\caption{Ambrosio-Tortorelli,\\
			$\gamma=\protect\input{Partitioning/gamma_AT.tex}$,
			$\alpha=\protect\input{Partitioning/alpha_AT.tex}$,
			$\varepsilon=\protect\input{Partitioning/eps_AT.tex}$
		}
	\end{subfigure}\hfill
	\begin{subfigure}{0.32\textwidth}
		\includegraphics[width=\textwidth]{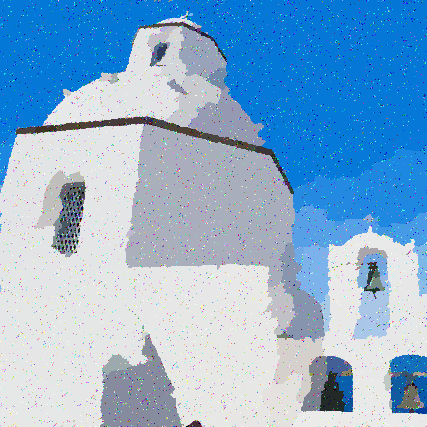}
		\caption{$L_0$ gradient smoothing \cite{xu2011image},
			$\kappa=\protect\input{Partitioning/kappa.tex}$,
			$\nu =\protect\input{Partitioning/lambda.tex}$			
		}
	\end{subfigure}	
	\captionsetup{justification=justified}
	\caption{\label{fig:segmentation}
		Comparison of partitionings of a natural image corrupted by Gaussian noise with
		$\sigma=\protect\input{Partitioning/sigma.tex}$.
		The Ambrosio-Tortorelli result is noisy and corrupted by clutter. 
		The $L_0$ gradient smoothing over-segments the large window on the
		left hand side, while details of the cross in the bottom right are smoothed out.
		The proposed result is visually competitive with the
		state-of-the-art graph cuts result.
	}
\end{figure}
Finally, we consider
the classical Potts problem
corresponding to $A=\mathrm{id}$ in \eqref{eq:pottsGeneralA}.
While the focus of the present paper is on 
a general imaging operator $A$, we next observe that it also works 
rather well for $A=\mathrm{id}$.
We used the full coupling scheme and set the step-size parameter to $\lambda = 0.55$.

To put our result in context we added the results of 
two other methods for $A=\mathrm{id}$:
the $L_0$ gradient smoothing
method of Xu et. al \cite{xu2011image} and
the state-of-the-art $\alpha$-expansion graph cut algorithm based on max-flow/min-cut 
of the
library GCOptimization 3.0 of Veksler and Delong
\cite{boykov2001fast,boykov2004experimental,kolmogorov2004energy}.
The method of \cite{xu2011image} has a parameter $\kappa>1$ to control
the convergence speed and a smoothing weight $\nu$.
In our experiments, we set $\kappa = 1.01$ and
$\nu=0.1$.
For the graph cuts
the same neighborhood weights
and jump penalty 
were used as for the proposed method.
The discrete labels are computed via k-means.

In Figure~\ref{fig:segmentation}, we show the results for a natural image corrupted by Gaussian noise.
The Ambrosio-Tortorelli result suffers from clutter and remains noisy.
The result of $L_0$ gradient smoothing over-segments
the textured window area while it smooths
out details of the cross.  
The state of the art graph cuts method and the proposed method both provide 
satisfying results which are
visually comparable.
Further, they yield solutions with comparable Potts energy values.
For instance, on the IVC dataset \cite{ivcdb} which consists
of $10$ natural color images of size $512\times 512,$
for the model parameters, $\gamma= 0.25$ and $\gamma=1,$
the mean values 
of the proposed approach are $7107.8$ and $13053.2$ compared with  	
the respective mean energies of the graph cut approach 
$7093.2$ and $13008.7$ which differ by about half a percent.
(For the results in Figure~\ref{fig:segmentation}, the energy value of the proposed approach is 25067.7
~\unskip compared with 
25119.5
~\unskip
for the graph cuts approach.)
Here, for the graph cut approach, we took the mean value of the input image on each computed segment before computing the Potts objective function.
To sum up, while the proposed method can handle general linear operators $A$,
the quality of the results for $A=\rm id$ is comparable 
with the state-of-the-art graph cut algorithm for $A=\rm id$.

\section{Conclusion}\label{sec:Conclusion}
In this paper, we have proposed a new iterative minimization strategy for multivariate piecewise constant Mumford-Shah/Potts energies as well as their quadratic penalty relaxations. 
Our schemes are based on majorization-minimization 
or forward-backward splitting methods of Douglas-Rachford type \cite{lions1979splitting}.
In contrast to the approaches in \cite{fornasier2010iterative,blumensath2008iterative,lu2012iterative, lu2013sparse} for sparsity problems
which lead to thresholding algorithms, our approach leads to non-separable yet computationally tractable problems in the backward step.

As a second part, we have provided a convergence analysis for the proposed algorithms. 
For the proposed quadratic penalty relaxation scheme, we have shown convergence towards a local minimizer. 
Due to the NP hardness of the quadratic penalty relaxation, the convergence result is in the range of what can be expected best. Concerning the scheme for the non-relaxed Potts problem we have also performed a convergence analysis. In particular, we have obtained results on the convergence towards local minimizers on subsequences. 
The quality of these results is comparable with the results of \cite{lu2012iterative, lu2013sparse} where, compared with these papers, 
we had to deal with the non-separability of the backward step as an additional challenge.
 
Finally, we have shown the applicability of our schemes in several experiments. 
We have applied our algorithms to deconvolution problems
including the problem of deblurring and denoising motion blur images.
We have further dealt with noisy and undersampled Radon data for the task of joint reconstruction, denoising and segmentation.
Finally, we have applied our approach to the situation of pure image partitioning (without blur)
which is a widely considered problem in computer vision.

\section*{Acknowledgements}
L. Kiefer and A. Weinmann were supported by the German Research Foundation (DFG) Grant WE5886/4-1. 
Additionally, A. Weinmann was supported by DFG Grant WE5886/3-1.
M. Storath was supported by DFG Grant STO1126/2-1.

{	\small
	\bibliographystyle{plain}
	\bibliography{IterativePottsMultivariate}
}

\end{document}